\documentclass[a4paper, cleveref, autoref, thm-restate]{lipics-v2021}

\title{The Distribution of the Deepest Leaves in Binary Trees}
\titlerunning{The Distribution of the Deepest Leaves in Binary Trees}

 \author{Olivier Bodini}{Sorbonne Paris Nord, LAGA}{Olivier.Bodini@univ-paris13.fr}{https://orcid.org/0000-0002-1867-667X}{}
 \author{Antoine Genitrini}{Sorbonne Universit\'e, CNRS, LIP6, F-75005 Paris, France}{Antoine.Genitrini@lip6.fr}{https://orcid.org/0000-0002-5480-0236}{}
 \author{Khaydar Nurligareev}{Université Bourgogne Europe, LIB}{Khaydar.Nurligareev@lip6.fr}{https://orcid.org/0000-0002-5687-2882}{}

 \funding{\textsc{anr-fwf} project \textsc{PAnDAG} ANR-23-CE48-0014-01}

 \authorrunning{O. Bodini, A. Genitrini, and K. Nurligareev}

 \Copyright{Olivier Bodini, Antoine Genitrini, and Khaydar Nurligareev} 



\nolinenumbers

\usepackage[utf8]{inputenc}
\usepackage[T1]{fontenc}
\usepackage{hyperref}
\usepackage{url}
\usepackage{booktabs}
\usepackage{amsfonts}
\usepackage{microtype}
\usepackage{amsmath,amsthm,amssymb}
\usepackage{stmaryrd}
\usepackage{mathtools}
\usepackage{graphicx}
\usepackage{wrapfig}
\usepackage{doi}

\newcommand{\Qone}{\mathrm Q}
\newcommand{\Qtwo}{\hat{\mathrm Q}}

\renewcommand{\Re}{\operatorname{Re}}
\renewcommand{\Im}{\operatorname{Im}}

\usepackage{pgfplots}
\pgfplotsset{compat=1.16}
\usepackage{array,multirow,makecell,diagbox}
\setcellgapes{1pt}
\makegapedcells
\newcolumntype{R}[1]{>{\raggedleft\arraybackslash }b{#1}}
\newcolumntype{L}[1]{>{\raggedright\arraybackslash }b{#1}}
\newcolumntype{C}[1]{>{\centering\arraybackslash }b{#1}}

\date{}

\newtheorem{problem}{Problem}

\begin{document}

\maketitle

\begin{abstract}
We study the extreme local structure of plane binary trees through the limiting distribution of the number of leaves at maximum depth.  For every $m\geq1$, we determine the limiting probability $\kappa^{[m]}$ that the deepest level contains $2m$ leaves.  The resulting sequence is encoded by a generating function $K(y)$, obtained from a bivariate refinement of the Catalan iteration.

We show that, at the critical point of this iteration, a parabolic recurrence governs the marked generating functions. Using continuous iteration theory and a Fatou coordinate, we transform this recurrence into the functional equation
$K((1+y/4)^2)=K(y)+4$,
which describes the limiting law.  Its singular behavior yields the exponential tail $\kappa^{[m]}\sim4^{-m+1}$.
Numerical evaluation gives, e.g., $\kappa^{[1]}\approx0.7009$, $\kappa^{[2]}\approx0.2226$ and an average number of deepest leaves equal to $\hat{\kappa}=4\kappa^{[1]}\approx2.8037$.
\end{abstract}


\newpage

\section{Introduction}

Binary trees are among the most fundamental objects in theoretical computer science, underlying classical data structures such as search trees, heaps, and priority queues. While their global parameters such as height, path length, and profile have been extensively studied, much less is known about the extreme local structure near the boundary of large random trees.
A natural parameter of interest is the number of leaves at maximum depth. It characterizes the microscopic shape of large random trees, providing essential insights into the behavior of worst-case search paths and the saturation of tree-based architectures.

This paper extends the analytic study of tree height initiated by Flajolet and Odlyzko~\cite{FlajoletOdlyzko1982}.  Following the principles of the \emph{symbolic method}~\cite[Part A]{FS09}, we use generating functions to encode the recursive construction of trees, and singularity analysis~\cite[Part B]{FS09} to extract asymptotic information.  The parameter considered here lies outside the most customary analytic-combinatorial schemes: the number of leaves at maximum depth is neither additive nor determined by a fixed local toll function.  It depends simultaneously on the global height and microscopic configuration at the boundary of the tree.

For this reason, the standard treatment of additive parameters does not directly apply.  Instead, we introduce a quasi-logistic framework for the truncation error in the Catalan iteration.  A three-zone analysis of its critical scaling regime provides the estimates needed to retain the information carried by the last level.  This approach leads to the complete limiting distribution of the number of leaves at the deepest level and, as a consequence, to high-precision constants for the number of leaves it contains and ultimately to an algebraic identity between the generating functions associated with the probability of having exactly two deepest leaves and the average number of leaves at the last level.

The global enumeration properties of plane binary trees, governed by the Catalan equation $C(z)=1+zC(z)^2$ and the dominant singularity $\rho=1/4$, are well established in literature.
Beyond the now-classical asymptotic counting, the geometric structure of large random trees has been the subject of an abundant literature, largely centered on the universal limit behavior of these objects.
Pioneering papers on typical height~\cite{FlajoletOdlyzko1982} and tree profiles~\cite{DrmotaGittenberger1997} (synthesized in~\cite{Drmota2009, FS09}) have shown that, once normalized, the contour of a random plane tree (omitting its leaves) converges to the Brownian excursion.
In this probabilistic framework, notably developed through Aldous's theory of the Continuum Random Tree~\cite{Aldous1, Aldous2, Aldous3}, the depth of the leaves is intrinsically linked to the local time.

These scaling limits provide a powerful description of the tree's envelope.  However, the usual rescaling that leads to the Brownian excursion collapses the last level into the maximum of a continuous path.  The finite number of leaves attaining this maximum is therefore lost in the limit and cannot be recovered directly from the Brownian description. This is precisely the information that our discrete analysis preserves.

We consequently shift the focus from the continuous global limit back to the dynamics of the Catalan iterates near their parabolic fixed point.  The quasi-logistic recurrence for the truncation error resolves the microscopic fluctuations that remain invisible after Brownian rescaling.  Note that by bridging dynamic growth and static unlabeled structures, this approach also unveils links with polynomial iteration and Mandelbrot polynomials that do not appear in purely continuous limit models; see~\cite{BGN26}.

While the main mass of the nodes is located at heights proportional to $\sqrt{n}$, the local structure near the deepest levels remains a delicate area subject to strong microscopic fluctuations.
The intricate dynamics of these extreme tails, as well as the collapse of probabilities and phase transitions in these tail regions, have been explored notably by Hwang and his co-authors~\cite{Hwang2005, FuchsHwangNeininger2006}, emphasizing the need for very fine scaling methods.

Our main objective is to determine the limiting distribution of the width of the deepest level.  Let $X(T)$ be the number of leaves at maximum depth in a tree $T$, and let $\mathcal{T}_n$ be the set of rooted plane binary trees with $n$ internal nodes.  Since the deepest leaves occur in pairs, we study the limits
\[
    \kappa^{[m]}=\lim_{n\to\infty}
    \frac{\bigl|\{T\in\mathcal{T}_n:X(T)=2m\}\bigr|}{|\mathcal{T}_n|}
    \qquad (m\geq1).
\]

\begin{wrapfigure}{r}{0.42\textwidth}
     \vspace{-\baselineskip}
     \centering
        \resizebox{\linewidth}{!}{%
        \begin{tabular}{|c|cccc||c|} 
            \hline 
            \diagbox{$n$}{$2m$} & 2 &  4 &  6 &  8 & $|\mathcal{T}_n|$ \\ 
            \hline
            1 & 1 &  &  &  & 1 \\
            2 & 2 &  &  &  & 2 \\
            3 & 4 & 1 &  &  & 5\\
            4 & 12 & 2 &  &  & 14\\
            5 & 32 & 10 &  &  & 42\\
            6 & 104 & 24 & 4 &  & 132\\
            7 & 328 & 92 & 8 & 1 & 429\\
            8 & 1080 & 308 & 40 & 2 & 1430\\
            9 & 3648 & 1028 & 176 & 10 & 4862\\
            10 & 12544 & 3584 & 584 & 84 & 16796\\
            \hline
        \end{tabular}%
        }
        \caption{Number of trees of size $n$ with $2m$ leaves at the deepest level (first values).}
        \label{fig:enum}
     \vspace{-\baselineskip}
\end{wrapfigure}

We derive a functional equation for the gene\-rating function of the sequence $(\kappa^{[m]})_{m\geq1}$, show that $\sum_m\kappa^{[m]}=1$, and deduce the exponential decay $\kappa^{[m]}=4^{-m+1}\bigl(1-\frac1{4m}+O(m^{-2})\bigr)$.  The probability of the minimal configuration with two deepest leaves and the mean of $X(T)$ are recovered as consequences of this distributional result; the case of two deepest leaves, on which the whole analysis is built, is treated in Appendix~\ref{app:two-problems}.

The sequence $\big(|\mathcal{T}_n|\big)_{n\ge0}$ is the well-known se\-quence of Catalan numbers, which are ubiquitous in combinatorics and appear in countless contexts. 
It is listed as~\href{https://oeis.org/A000108}{\texttt{OEIS A000108}}\footnote{OEIS stands for the On-Line Encyclopedia of Integer Sequences.}.
At the same time, the decomposition of the number of trees according to the width of their deepest level does not seem to be recorded in OEIS; see the table in Figure~\ref{fig:enum}.  The columns illustrate that the natural question concerns the entire width distribution, of which the two-leaf configuration is only the first case.
Note that the empty cells in the table correspond to a zero value; furthermore, we omit the degenerate case of the tree reduced to a single leaf.

Let us briefly explain why this machinery is needed to reach $\kappa^{[m]}$.
First, the deepest level sits at a random height, of order $\sqrt{n}$, so
one cannot condition on a fixed depth: all heights must be handled
simultaneously. Second, the dominant contribution comes from the
double-scaling regime where the height $k$ is of order $1/\tau$, with
$\tau=\sqrt{1-4z}$, which is exactly where the truncation error of the
Catalan iteration is neither negligible nor saturated. Third, within this
regime one must keep track of the discrete configuration of the last level,
which the Brownian rescaling discards. Finally, marking the leaves of the
deepest level turns these three requirements into a single near-parabolic
recurrence, whose Fatou coordinate yields the functional equation of
Section~\ref{sec:generalization}.

The paper is organized as follows.  Section~\ref{sec:analysis} introduces the Catalan iteration and establishes the estimates for its truncation error. Section~\ref{sec:generalization} uses this framework to obtain the complete limit distribution through a bivariate generating function and a Fatou coordinate. Section~\ref{sec:results-perspectives} presents numerical values and perspectives. Appendix~\ref{app:two-problems} contains the two original special cases and their refined critical analysis.

\section{Generating functions for trees up to a given height}
\label{sec:analysis}

In this section, we introduce the classical combinatorial framework and establish the necessary estimates for the truncation error,
that is, the generating function of the trees whose height exceeds a given level $k$
(i.e., the trees still missing when the recursive construction is stopped at height $k$).
We analyze the behavior of the error function at the critical point and describe the scaling regime that governs its dynamics in a neighborhood of this point.

As a first step, we decompose trees according to their height.
Given an integer $k\in\mathbb{N}$, denote by $I_{k}(z)$ the generating function of plane binary trees of height at most $k$.
Here, $z$~is the marking variable for the internal nodes.
The generating functions $I_{k}(z)$ are actually polynomials and satisfy the recurrence
\begin{equation}\label{eq:L-def}
    I_0(z)=1, \qquad 
    I_{k+1}(z)=1+z \cdot I_{k}(z)^2 \quad 
    \mbox{for } k \ge 0.
\end{equation}
The combinatorial interpretation of recurrence~\eqref{eq:L-def} is the following: to get a tree of height at most $k+1$, we take two subtrees of height at most $k$ and attach them to the root.

\subsection{Definitions and notations}
\label{sec:def-not}

Our study relies on the generating function of plane binary trees, commonly called the Catalan generating function and denoted by $C(z)$. 
According to the classical symbolic method developed in analytic combinatorics~\cite{FS09}, this series is derived in the following way.
A tree is decomposed into atoms representing nodes; we represent an atom using the formal variable $\mathcal{Z}$.
The combinatorial class $\mathcal{T}$ of plane binary trees then satisfies the relation
$\mathcal{T} = \mathcal{E} \cup \mathcal{Z} \times \mathcal{T}^2$,
combinatorially interpreted as follows: a tree in $\mathcal{T}$ is either the empty tree $\mathcal{E}$ (the tree reduced to a single leaf) or a root node~$\mathcal{Z}$ with two attached subtrees of $\mathcal{T}$.
This formal equation is translated into the context of generating functions through the functional equation $C(z)=1+z\cdot C(z)^2$, whose resolution yields the explicit form:
\[
    C(z)=\frac{1-\sqrt{1-4z}}{2z}.
\]
The asymptotic behavior of the sequence $\big(|\mathcal{T}_n|\big)_{n\ge0}$ is determined by the radius of convergence $\rho=1/4$ of the generating function $C(z)$.
At the singularity, it takes the value $C(\rho)=2$.
To analyze the behavior in a neighborhood of this dominant singularity, it is particularly convenient to introduce the singular variable 
\[
    \tau=\sqrt{1-4z} \in [0,1) 
    \qquad\mbox{for } z \in (0,\rho],
\]
as well as an auxiliary quantity $r(z)=2zC(z)=1-\tau$.

To understand the structure of the deepest levels of a tree, it is natural to study families of truncated trees.
Indeed, the set of all plane binary trees can be approximated by the set of trees of height at most $k$ encoded by the polynomial $I_{k}(z)$ satisfying~\eqref{eq:L-def}.
Thus, we study the truncation error
\begin{equation}\label{eq:e-def}
    e_k(z) := C(z)-I_{k}(z).
\end{equation}
From the combinatorial meaning, $e_k(z)>0$ for all $z \in (0,\rho]$.
By simple subtraction and factorization, one can see that this truncation error obeys the quasi-logistic recurrence (observed in the context of dynamical systems;
see the book by Strogatz~\cite{strogatz2015nonlinear}):
\begin{equation}\label{eq:e-recur}
    e_0(z)=C(z)-1, \qquad
    e_{k+1}(z)=r(z)e_k(z)-ze_k(z)^2 \quad
    \mbox{for } k \ge 0.
\end{equation}
In the following, the functions $e_k(z)$ will be key quantities in our analysis.

\subsection{Asymptotic analysis of the truncation error}

In this subsection, we analyze the asymptotic behavior of the truncation error $e_k(z)$.
The result presented here concerns its behavior at the critical point $\rho = 1/4$ and is based on the classical literature on tree heights \cite{FlajoletOdlyzko1982,CDJ2001}.
The technical tools are based on analytic combinatorics; see Chapter~VI of the book~\cite{FS09}.
Note that we need a deeper expansion than that provided in the cited papers.
The proof of the proposition below is detailed in Appendix~\ref{app:analysis}.

\begin{proposition}[Classical behavior at the critical point~\cite{FlajoletOdlyzko1982}]
\label{thm:e-asymp}
    At $z=\rho$, the truncation error admits the following asymptotic expansion:
    \[
        e_k(\rho) \underset{k\to\infty}{=} \frac{4}{k} - \frac{4\log k}{k^2} + \frac{C}{k^2} + O\!\left(\frac{\log^2 k}{k^3}\right),
    \]
    where $C$ is a constant depending on the initial value $e_0(\rho)$.
\end{proposition}


\subsection{Near-parabolic scaling of the truncation error}
\label{sec:near-parabolic}

In this subsection, we describe the asymptotic behavior of the
truncation error $e_k(z)$ in the double-scaling regime where
$u_k:=k\tau$ stays bounded.
The proofs are given in Appendix~\ref{app:analysis}.

\noindent Throughout this section, $\varepsilon>0$ is a small parameter,
$\tau=\sqrt{1-4z}$ denotes the branch with $\Re\tau\ge0$,
$U$ is a positive real number defined below,
and $z$ ranges over one of the following two sets:
\begin{enumerate}
    \item[(S1)] a $\Delta$-domain
    $\Delta_\varepsilon := \{ z : 0<|z-\rho|<\varepsilon,\ |\arg(z-\rho)|>\theta \}$
    with $\theta\in(0,\pi/2)$; in this case, $|\arg\tau|\le\psi:=(\pi-\theta)/2<\pi/2$,
    and $U>0$ is arbitrary;
    \item[(S2)] the punctured disc $\{ z : 0<|z-\rho|<\varepsilon\}$; in this case,
    $|\arg\tau|\le\psi:=\pi/2$, and $U\in(0,6]$.
\end{enumerate}
Setting (S2) is only used, through Cauchy's estimates, to control small
heights in Appendix~\ref{app:proofs-sec3}; the reader may focus on (S1).
In both cases, all points $u_j=j\tau$ with $1\le j\le U/|\tau|$ lie in the
compact set
$\mathcal{S}_U := \{u\in\mathbb{C} : 0\le|u|\le U,\ |\arg u|\le\psi\}\cup\{0\}$.
The zeros of $e^u-1$ are the points $2i\pi n$, $n\in\mathbb{Z}$, and none of
them lies in $\mathcal{S}_U\setminus\{0\}$ (in case (S2) because $U<2\pi$).
Hence the function 
    \begin{equation}\label{eq:def-eta}
      \eta(u):=\dfrac{e^u-1}{u},
      \qquad \mbox{with } \eta(0)=1,
    \end{equation}
satisfies $\eta_*\le|\eta(u)|\le\eta^*$ on $\mathcal{S}_U$ for some constants
$0<\eta_*\le1\le\eta^*$. Finally, we take $\varepsilon$ small enough so that
$|z|\le\frac12$ and $|\alpha z|\le\frac38$, where $\alpha:=1/(1-\tau)$; this
is possible since $\alpha z\to\rho=\frac14$ as $z\to\rho$. All constants
below depend only on $\theta$ and~$U$.

\noindent\emph{Heuristics.} At $z=\rho$ the error $e_k(\rho)$ decays slowly, like
$4/k$ (Proposition~\ref{thm:e-asymp}). For a real $z$ that is slightly below $\rho$, the
linear coefficient $r(z)=1-\tau$ in~\eqref{eq:e-recur} is smaller than~$1$,
so the error eventually decays geometrically, like $e^{-k\tau}$. The
transition between the two behaviors occurs when $k$ is of order $1/\tau$.
The inverse $1/e_k(z)$ makes this explicit: at each step it gains about $1/4$,
as at the critical point, and is also multiplied by about $1+\tau$, so that
$1/e_k(z)\approx(e^{k\tau}-1)/(4\tau)$. This value grows like $k/4$ when $k\tau$ is
small, and it grows exponentially when $k\tau$ is of order~$1$. The two lemmas below
make this picture rigorous, also for complex~$z$.

\begin{lemma}[Complex barrier for the Riccati variable]
\label{lem:barrier-v}
    In the above settings, there exist constants $c_0,C_0,C_1,\tau_0>0$ such
    that for every $z$ with $0<|\tau|<\tau_0$ and every integer $j$ with
    $1\le j\le U/|\tau|$,
    \begin{equation}\label{eq:barrier-v}
        c_0\,j \;\leq\; |v_j(z)| \;\leq\; C_0\,j
        \qquad\text{and}\qquad
        \left|v_j(z)-\frac{e^{u_j}-1}{4\tau}\right| \;\leq\; C_1\log(j+1),
    \end{equation}
    where $v_j(z)=1/e_j(z)$. In particular, $e_j(z)\neq0$ on this
    range and $|e_j(z)|=\Theta(1/j)$.
\end{lemma}

Lemma~\ref{lem:barrier-v} is the a~priori estimate on which the whole
analysis rests: on the time scale $j\lesssim 1/|\tau|$, the perturbation
$\tau$ has not yet had time to push $v_j(z)=1/e_j(z)$ away from the linear growth
of the critical case, so that $e_j(z)$ does not vanish and remains of
order $1/j$, even for complex~$z$.

\begin{lemma}[Near-parabolic scaling]
\label{lem:near-parabolic-scaling}
    In the above settings, uniformly for $0<|\tau|<\tau_0$ and $1\le k\le U/|\tau|$,
    where $\tau_0$ is determined by Lemma~\ref{lem:barrier-v},
    \begin{equation}\label{eq:near-parabolic-main}
        e_k(z) = \tau\,\phi(u_k) + O\!\left(\frac{\log k}{k^2}\right),
        \qquad
        \phi(u) := \frac{4}{e^u-1} = 2\left(\coth\!\left(\frac{u}{2}\right)-1\right),
    \end{equation}
    and, uniformly for $2\le k\le U/|\tau|$,
    \begin{equation}\label{eq:factor-expansion}
        2z\,I_{k-1}(z) = 1-\tau\coth\!\left(\frac{u_k}{2}\right)
        + O\!\left(\frac{\log k}{k^2}+|\tau|^2\right).
    \end{equation}
    In particular, if $u_k$ stays in a compact subset of
    $\mathcal{S}_U\setminus\{0\}$, then $k=\Theta(1/|\tau|)$, both error terms
    are $O\big(|\tau|^2\log\frac{1}{|\tau|}\big)$, and $|e_k(z)|=\Theta(|\tau|)$.
\end{lemma}

In other words, near the singularity the truncation error depends on
$k$ and $z$ only through the scaling variable $u_k=k\tau$: it is $\tau$
times the inverse profile $\phi$ of the linearized solution, the
nonlinear effects contributing only to the error term.

\section{Distribution of the number of deepest leaves}
\label{sec:generalization}

In this section, we determine the complete probability distribution of the number of leaves at maximum depth. 
To do so, we distinguish trees according to the number of leaves reaching this level. Since every internal 
node at the preceding level produces two deepest leaves, this number is necessarily even.

\subsection{Combinatorial recurrence and leading scaling}
\label{sec:cascade}

Given two integers $k\in\mathbb{Z}_{\ge0}$ and $m\in\mathbb{Z}_{\ge1}$, 
let $Q_k^{[m]}(z)$ be the generating function of plane binary trees of height $k$ with exactly $m$ internal nodes
at the $(k-1)$th level (and thus with $2m$~leaves at the $k$th level being the deepest).
These $m$~internal nodes are distributed between the left and right subtrees with respect to the root decomposition, which leads to the convolution
\begin{equation}\label{eq:Q^[m]-def}
    Q_{k+1}^{[m]}(z) = 2zI_{k-1}(z)Q_k^{[m]}(z)
    + z \sum_{i=1}^{m-1} Q_k^{[i]}(z) \cdot Q_k^{[m-i]}(z)
    \quad \mbox{for } k \ge 1,
\end{equation}
accompanied by the initial conditions
\[
    Q_0^{[1]}(z) = 0, \qquad
    Q_1^{[1]}(z) = z, \qquad
    Q_0^{[m]}(z) = Q_1^{[m]}(z) = 0
    \quad \mbox{for } m \ge 2.
\]
For every $m\ge1$, we also introduce the generating series of plane binary trees whose
deepest level contains exactly $2m$ leaves:
\begin{equation}\label{eq:G^[m]-def}
    G^{[m]}(z) \;:=\; \sum_{k\ge1} Q_k^{[m]}(z).
\end{equation}
Note that, for every $n\ge1$, $[z^n]\,G^{[m]}(z) = \bigl|\{T\in\mathcal{T}_n : X(T) = 2m\}\bigr|$,
where $\mathcal{T}_n$ denotes the set of plane binary trees with $n$
internal nodes and $X(T)$ is the number of leaves of maximal depth in $T$.
The first term on the right-hand side of~\eqref{eq:Q^[m]-def} isolates the cases where all $m$ nodes belong to a single (sufficiently deep) subtree.
Its specialization to $m=1$ is revisited independently in Appendix~\ref{app:two-problems}.
The convolution sum accounts for the complementary situation in which both subtrees reach the maximum height and share the $m$ deepest internal nodes.

For every $m\ge1$, define the limiting amplitude
\begin{equation}\label{eq:Km-def}
    K^{[m]} \;:=\; \lim_{k\to\infty} k^2\,Q_k^{[m]}(\rho).
\end{equation}
The existence and positivity of this limit are part of
Lemma~\ref{lem:marked-scaling} below.

\emph{The case $m=1$.}
For $m=1$, the convolution term vanishes, and~\eqref{eq:Q^[m]-def} gives
the product formula
\begin{equation}\label{eq:prod_Qk}
    Q_k^{[1]}(z) = z \prod_{j=1}^{k-1} 2z I_{j-1}(z) \qquad (k\ge1).
\end{equation}
At the critical point $\rho=\frac{1}{4}$, we have $2\rho I_{j-1}(\rho)=1-\frac12e_{j-1}(\rho)=1-\frac2j+O(\log j/j^2)$ by
Proposition~\ref{thm:e-asymp};
this yields $Q_k^{[1]}(\rho)=K^{[1]}/k^2+O(\log k/k^3)$ with $K^{[1]}>0$
(Lemma~\ref{lem:Qk-rho} in Appendix~\ref{app:two-problems}). Near $\rho$,
the factors in~\eqref{eq:prod_Qk} are described by relation~\eqref{eq:factor-expansion}. 
The comparison of the products corresponding to these two regimes is the engine of the whole analysis.
Set
\begin{equation}\label{eq:psi-def}
    \psi(u) := \frac{u^2e^u}{(e^u-1)^2} = \left(\frac{u/2}{\sinh(u/2)}\right)^2,
    \qquad
    \vartheta(s) := \coth\frac s2-\frac2s .
\end{equation}
Both functions are analytic in a neighborhood of $\mathcal{S}_U$ (their
only singularities are at the points $2i\pi n$, $n\neq0$), with
$\psi(0)=1$, $\vartheta(0)=0$ and $\vartheta(s)=s/6+O(s^3)$. Moreover,
integrating along the segment $[0,u]\subseteq\mathcal{S}_U$ gives
\begin{equation}\label{eq:int-theta}
    \int_0^u\vartheta(s)\,ds = 2\log\frac{2\sinh(u/2)}{u} = -\log\psi(u),
\end{equation}
since $2\log\sinh(s/2)$ and $2\log s$ are primitives of $\coth(s/2)$ and
$2/s$, respectively, while $2\sinh(s/2)/s\to1$ as $s\to0$.

\begin{lemma}\label{lem:Q-comparison}
In both settings (S1) and (S2), as $z\to\rho$,
\[
    Q_k^{[1]}(z) = Q_k^{[1]}(\rho)\,\psi(u_k)\bigl(1+o(1)\bigr),
    \qquad\text{uniformly for } 1\le k\le U/|\tau| .
\]
\end{lemma}

\noindent\emph{Key idea.} By~\eqref{eq:prod_Qk}, the ratio
$Q_k^{[1]}(z)/Q_k^{[1]}(\rho)$ is the product of the ratios of the factors
$2zI_{j-1}(z)$ and $2\rho I_{j-1}(\rho)$. At the critical point,
Proposition~\ref{thm:e-asymp} gives $2\rho I_{j-1}(\rho)=1-\frac2j+O(\log j/j^2)$,
while near $\rho$, relation~\eqref{eq:factor-expansion} gives
$2zI_{j-1}(z)=1-\tau\coth(u_j/2)+O(\log j/j^2+|\tau|^2)$. Since
$\tau\coth(u_j/2)=\frac2j+\tau\vartheta(u_j)$, the singular parts $\frac2j$
cancel, and the logarithm of the ratio is, up to small errors, the Riemann
sum $-\tau\sum_{j<k}\vartheta(j\tau)\approx-\int_0^{u_k}\vartheta(s)\,ds=\log\psi(u_k)$
by~\eqref{eq:int-theta}. The delicate point of the reasoning is that the errors
$O(\log j/j^2)$ are not small for bounded $j$; these finitely many initial
factors are handled by continuity at $z=\rho$. The complete proof, together
with the consequences of the lemma for $m=1$ (uniform bounds, scaling
profile and exponential tail), is given in Appendix~\ref{sec:Q-comparison}.

\emph{General case.}
The first term of~\eqref{eq:Q^[m]-def} is linear in $Q_k^{[m]}$(z), with the
same factor $2zI_{k-1}(z)$ as in the case $m=1$. This suggests factoring
out $Q_k^{[1]}(z)$: for $m,k\ge1$, we set
\begin{equation}\label{eq:duhamel}
    W_k^{[m]}(z) := Q_k^{[m]}(z)/Q_k^{[1]}(z)
    \qquad (\text{whenever } Q_k^{[1]}(z)\neq0),
\end{equation}
so that $W_k^{[1]}(z)=1$. Dividing~\eqref{eq:Q^[m]-def} by
$Q_{k+1}^{[1]}(z)=2zI_{k-1}(z)Q_k^{[1]}(z)$ yields the exact
\emph{variation-of-constants} (or Duhamel) formula for $m\ge2$:
\begin{equation}\label{eq:W-def}
    W_1^{[m]}(z)=0,
    \qquad
    W_{k+1}^{[m]}(z) = W_k^{[m]}(z)
    + \frac{Q_k^{[1]}(z)}{2I_{k-1}(z)}
      \sum_{i=1}^{m-1} W_k^{[i]}(z)\,W_k^{[m-i]}(z)
    \qquad (k\ge1).
\end{equation}
The whole hierarchy is thus driven by the case $m=1$: the increments of
$W_k^{[m]}(z)$ are proportional to $Q_k^{[1]}(z)=O(1/k^2)$, hence summable.
This is the key to the following lemma.

\begin{lemma}[Scaling of the marked generating functions]
\label{lem:marked-scaling}
Let $m\ge1$, and let $\Delta$ be a sufficiently small $\Delta$-domain at $\rho=1/4$.
In this case, the following holds.
\begin{enumerate}[\upshape(i)]
    \item At the critical point, $W_k^{[m]}(\rho)$ increases to a finite
    limit $\omega_m>0$ (with $\omega_1=1$). The limit~\eqref{eq:Km-def}
    exists, $K^{[m]}=\omega_m K^{[1]}>0$, and
    $Q_k^{[m]}(\rho) = K^{[m]}/k^2 + O(\log k/k^3)$.
    \item There exist constants $C_m,A_m,c,U_0>0$ such that, for all
    $z\in\Delta$, $|Q_k^{[m]}(z)|\le C_m/k^2$ for every $k\ge1$, and
    $|Q_k^{[m]}(z)|\le A_m|\tau|^2e^{-ck|\tau|}$ for every $k$ with
    $k|\tau|\ge U_0$.
    \item Let $\mathcal{K}\subseteq\{u\in\mathbb{C}:\Re u>0\}$ be compact.
    Uniformly for $u_k=k\tau\in\mathcal{K}$, as $z\to\rho$ in $\Delta$,
    \begin{equation}\label{eq:marked-scaling}
        Q_k^{[m]}(z)
        \;=\;
        \tau^2\,\Psi_m(u_k) + o\big(|\tau|^2\big),
        \qquad
        \Psi_m(u)
        \;=\;
        K^{[m]}\,\frac{e^{u}}{(e^{u}-1)^2}
        \;=\; \frac{K^{[m]}}{4\sinh^2(u/2)}.
    \end{equation}
\end{enumerate}
\end{lemma}

\begin{proof}
    We take $\Delta$ small enough for the results of Appendix~\ref{app:two-problems} to be applied,
    with the constants $U_0,c$ of Lemma~\ref{lem:Q-tail}.
    Also, we set $\bar\Delta:=\Delta\cup\{\rho\}$.

\emph{Preliminary facts for $m=1$.}
    By Corollary~\ref{cor:Q-bounds} (with $U=U_0$) and Lemma~\ref{lem:Q-tail},
    for all $z\in\bar\Delta$ and $k\ge1$, we have
    $|I_{k-1}(z)|\ge\frac12$, $|Q_k^{[1]}(z)|\le A'/k^2$, and $Q_k^{[1]}(z)\neq0$
    (for the latter inequality, note that for $k>U_0/|\tau|$, the factors $2zI_{k-1}(z)$ do not vanish since $|2zI_{k-1}(z)|\ge|1-\tau|-|2ze_{k-1}(z)|\ge1-2|\tau|$).
    In particular, the functions $W_k^{[m]}(z)$ are well defined and continuous on $\bar\Delta$,
    and $T:=\sup_{z\in\bar\Delta}\sum_{k\ge1}|Q_k^{[1]}(z)|<\infty$.

\emph{Uniform bounds on $W_k^{[m]}$.}
    Define $\bar\omega_1:=1$ and $\bar\omega_m:=T\sum_{i=1}^{m-1}\bar\omega_i\bar\omega_{m-i}$.
    By induction on~$m$, using~\eqref{eq:duhamel} and $1/\big(2|I_{k-1}(z)|\big)\le1$, we have
    \[
        \left|W_{k+1}^{[m]}(z)-W_k^{[m]}(z)\right| \le
        \left|Q_k^{[1]}(z)\right|\sum_{i=1}^{m-1}\bar\omega_i\bar\omega_{m-i},
        \qquad\text{hence}\qquad
        \left|W_k^{[m]}(z)\right|\le\bar\omega_m
    \]
    for all $k\ge1$ and $z\in\bar\Delta$.
    Moreover, the series of increments converges uniformly on $\bar\Delta$:
    $W_k^{[m]}\to W_\infty^{[m]}$ with $\left|W_\infty^{[m]}(z)-W_k^{[m]}(z)\right|\le C_m'\sum_{j\ge k}A'/j^2\le C_m''/k$.
    As a uniform limit of continuous functions, $W_\infty^{[m]}$ is continuous at $\rho$.
    Setting $\omega_m:=W_\infty^{[m]}(\rho)$, we obtain
    \begin{equation}\label{eq:W-conv}
        \sup_{k\ge k_1}\left|W_k^{[m]}(z)-\omega_m\right|
        \le \frac{C_m''}{k_1}+\left|W_\infty^{[m]}(z)-\omega_m\right|
        \xrightarrow[k_1\to\infty,\ z\to\rho]{} 0 .
    \end{equation}

\emph{Proof of (i).}
    At $z=\rho$, all the quantities in~\eqref{eq:duhamel} are nonnegative,
    so $W_k^{[m]}(\rho)$ is nondecreasing in $k$, converges to $\omega_m$,
    and $0\le\omega_m-W_k^{[m]}(\rho)\le C_m''/k$.
    By induction on $m$, $W_k^{[m]}(\rho)>0$ for $k\ge m$: this is clear for $m=1$,
    and for $m\ge2$ the increment from $k=m-1$ to $k=m$ contains the term
    $Q_{m-1}^{[1]}(\rho)\,W_{m-1}^{[1]}W_{m-1}^{[m-1]}(\rho)/(2I_{m-2}(\rho))>0$.
    Hence $\omega_m>0$. By Lemma~\ref{lem:Qk-rho},
    \[
        Q_k^{[m]}(\rho)=Q_k^{[1]}(\rho)\,W_k^{[m]}(\rho)
        =\left(\frac{K^{[1]}}{k^2}+O\Bigl(\frac{\log k}{k^3}\Bigr)\right)
         \left(\omega_m+O\Bigl(\frac1k\Bigr)\right)
        =\frac{\omega_mK^{[1]}}{k^2}+O\left(\frac{\log k}{k^3}\right),
    \]
    so that the limit~\eqref{eq:Km-def} exists and equals $K^{[m]}=\omega_mK^{[1]}>0$.

\emph{Proof of (ii).}
    We have $|Q_k^{[m]}(z)|=|Q_k^{[1]}(z)|\cdot|W_k^{[m]}(z)|\le\bar\omega_m|Q_k^{[1]}(z)|$,
    and the two bounds follow from Lemma~\ref{lem:Q-tail},
    with $C_m:=\bar\omega_mA'$ and $A_m:=\bar\omega_mA$.

\emph{Proof of (iii).}
    Since $\tau$ lies in the sector $|\arg\tau|\le(\pi-\theta)/2$,
    the points $u_k\in\mathcal{K}$ lie in a compact subset of $\mathcal{S}_U\setminus\{0\}$ for some $U$, and $k\ge r/|\tau|\to\infty$ for some $r>0$.
    By Corollary~\ref{cor:Q-bounds}(iii),
    $Q_k^{[1]}(z)=K^{[1]}\tau^2g(u_k)\big(1+o(1)\big)$ with $g(u):=e^u/(e^u-1)^2$,
    and by~\eqref{eq:W-conv}, $W_k^{[m]}(z)=\omega_m\big(1+o(1)\big)$.
    Since the function $g$ is bounded on $\mathcal{K}$, after multiplication we get
    $Q_k^{[m]}(z)=K^{[m]}\tau^2g(u_k)+o(|\tau|^2)$, which is~\eqref{eq:marked-scaling}.
\end{proof}

The leading profile is therefore \emph{universal in $m$}: the ratio
$Q_k^{[m]}(z)/Q_k^{[1]}(z)$ converges to the constant $\omega_m$, so that
$Q_k^{[m]}(z)$ inherits the profile of the case $m=1$, and only the
amplitude $K^{[m]}$ depends on~$m$. In the scaling window, the
convolution term of~\eqref{eq:Q^[m]-def} is of order $\tau^4$, hence
subleading. It is, however, \emph{not} negligible in the parabolic core
$k\ll1/|\tau|$: since $Q_1^{[m]}(z)=0$ for $m\ge2$, it is precisely this
term that creates the amplitude $\omega_m$, and hence $K^{[m]}$.

\subsection{Singularity of the marked generating functions}
\label{sec:singularity}

The goal of this section is to prove that, for every $m\ge1$, the generating series $G^{[m]}(z)$
admits the same dominant singularity $\sqrt{1-4z}$ as the Catalan series,
with a coefficient determined by the leading scaling profile
established in the previous section. Given
Lemma~\ref{lem:marked-scaling}, the argument is the same for every~$m$.

\begin{theorem}\label{thm:sqrt-singularity}
Let $m\ge1$. The series $G^{[m]}$ is analytic in a
$\Delta$-domain at $\rho=1/4$ and, as $z\to\rho$ within this $\Delta$-domain,
\begin{equation}\label{eq:sqrt-expansion}
    G^{[m]}(z)
    = G^{[m]}(\rho) - \frac{K^{[m]}}{2}\sqrt{1-4z}
      + o\!\left(\sqrt{1-4z}\right).
\end{equation}
Consequently,
\begin{equation}\label{eq:coeff-asympt}
    [z^n]G^{[m]}(z)
    \sim \frac{K^{[m]}}{4\sqrt{\pi}}\,4^n\,n^{-3/2}
    \qquad\text{and}\qquad
    \kappa^{[m]} = \frac{K^{[m]}}{4}.
\end{equation}
\end{theorem}

\noindent
Intuitively, at the critical point,
$Q_k^{[m]}(\rho)\sim K^{[m]}/k^2$, while near $\rho$ the height is
measured on the scale $1/|\tau|$ and
$Q_k^{[m]}(z)\approx K^{[m]}\tau^2g(k\tau)$, with
$g(u)=1/(4\sinh^2(u/2))\sim1/u^2$ as $u\to0$. The two descriptions agree
for $k\ll1/|\tau|$ and both terms are negligible for $k\gg1/|\tau|$, so
that $G^{[m]}(\rho)-G^{[m]}(z)$ is essentially a Riemann sum of step $\tau$:
\[
    G^{[m]}(\rho)-G^{[m]}(z)
    \approx K^{[m]}\,\tau\int_0^\infty\left(\frac1{u^2}-\frac1{4\sinh^2(u/2)}\right)\,du
    = \frac{K^{[m]}}{2}\,\tau .
\]
The proof below makes this rigorous by splitting the sum into three
zones; small heights are controlled by Cauchy's estimates
(Lemma~\ref{lem:small-zone}), and the analytic continuation of
$G^{[m]}$ to a~$\Delta$-domain is established in
Lemma~\ref{lem:G-analytic} (both in Appendix~\ref{app:proofs-sec3}).

\begin{proof}[Proof of Theorem~\ref{thm:sqrt-singularity}]
    By Lemma~\ref{lem:G-analytic}, $G^{[m]}$ is analytic in a $\Delta$-domain at $\rho$,
    and for $z\in\Delta$ we can write
    \[
        \Xi(z) := \frac{G^{[m]}(\rho)-G^{[m]}(z)}{\tau}
        = \frac{1}{\tau}\sum_{k\ge1}D_k,
        \qquad D_k := Q_k^{[m]}(\rho)-Q_k^{[m]}(z).
    \]
    We show that $\Xi(z)\to K^{[m]}/2$ as $z\to\rho$ in $\Delta$.
    Fix $0<\delta\le1$ and $\Lambda\ge U_0+1$, and set
    $M:=\lfloor\delta/|\tau|\rfloor$ and $N:=\lfloor\Lambda/|\tau|\rfloor$.

\emph{Small heights, $k\le M$.}
    Here, we have $k|\tau|\le\delta\le1$, and Lemma~\ref{lem:small-zone} gives us the inequality
    $\bigl|\frac1\tau\sum_{k\le M}D_k\bigr|\le M_m\,M|\tau|\le M_m\delta$.

\emph{Scaling zone, $M<k\le N$.}
    Here $u_k$ stays in the compact set $\mathcal{S}_\Lambda\cap\{|u|\ge\delta/2\}$.
    By Lemma~\ref{lem:marked-scaling}(i) and~(iii),
    $Q_k^{[m]}(\rho)=K^{[m]}\tau^2/u_k^2+o(|\tau|^2)$ and $Q_k^{[m]}(z)=\Psi_m(u_k)\tau^2+o(|\tau|^2)$ uniformly,
    so that $D_k=K^{[m]}\tau^2\varpi(u_k)+o(|\tau|^2)$ with $\varpi(u):=u^{-2}-e^u/(e^u-1)^2$.
    Summing the $N-M\le\Lambda/|\tau|$ error terms, we obtain
    \begin{equation}\label{eq:D-estim}
        \frac1\tau\sum_{M<k\le N}D_k
        = K^{[m]}\,\tau\!\!\sum_{M<k\le N}\varpi(k\tau) + o(1).
    \end{equation}
    The function $\varpi$ is analytic in a neighborhood of $\mathcal{S}_\Lambda$, with bounded derivative, and
    \[
        \Pi(u) := \frac12\coth\frac u2-\frac1u
    \]
    is a primitive of $\varpi$ (since $\frac{d}{du}\frac12\coth\frac u2=-1/4\sinh^2(u/2)$),
    with $\Pi(u)=u/12+O(u^3)$ as $u\to0$.
    The sum~\eqref{eq:D-estim} is a Riemann sum along the ray of direction $\arg\tau$, with step $\tau$, so
    \[
        \tau\!\!\sum_{M<k\le N}\varpi(k\tau)
        = \Pi(N\tau)-\Pi(M\tau)+O\bigl(\Lambda|\tau|\bigr).
    \]
    Now $|M\tau|\le\delta$ gives $|\Pi(M\tau)|\le C\delta$,
    while $|N\tau|\ge\Lambda-|\tau|$ and $\Re(N\tau)\ge\sigma(\Lambda-|\tau|)$
    (recall that $\Re\tau\ge\sigma|\tau|$ with $\sigma=\sin(\theta/2)$ in setting (S1))
    give $|\Pi(N\tau)-\frac12|\le\frac1{\Lambda-1}+2e^{-\sigma(\Lambda-1)}$,
    because $|\coth(u/2)-1|=2/|e^u-1|\le4e^{-\Re u}$ for $\Re u\ge1$.

\emph{Large heights, $k>N$.}
    Here, we have $k|\tau|>\Lambda-|\tau|\ge U_0$,
    and Lemma~\ref{lem:marked-scaling}(i)--(ii) give us $|D_k|\le C/k^2+A_m|\tau|^2e^{-ck|\tau|}$. 
    Therefore,
    \[
        \Bigl|\frac1\tau\sum_{k>N}D_k\Bigr|
        \le \frac{C}{|\tau|N}+A_m|\tau|\frac{e^{-cN|\tau|}}{1-e^{-c|\tau|}}
        \le \frac{2C}{\Lambda}+\frac{2A_m}{c}\,e^{-c(\Lambda-1)}
    \]
    for $|\tau|$ small, since $1-e^{-c|\tau|}\ge c|\tau|/2$.

\emph{Conclusion.}
    Combining the three zones,
    \[
        \limsup_{z\to\rho,\ z\in\Delta}\left|\Xi(z)-\frac{K^{[m]}}{2}\right|
        \le C\left(\delta+\frac1\Lambda+e^{-\sigma(\Lambda-1)}+e^{-c(\Lambda-1)}\right)
    \]
    for all $\delta\in(0,1]$ and $\Lambda\ge U_0+1$,
    with $C$ independent of $\delta$ and $\Lambda$.
    Letting $\delta\to0$ and $\Lambda\to\infty$ gives $\Xi(z)\to K^{[m]}/2$,
    which is~\eqref{eq:sqrt-expansion}.
    In particular, for $m=1$ this computes the integral
    $\int_0^\infty\bigl(u^{-2}-e^u/(e^u-1)^2\bigr)\,du=\Pi(+\infty)-\Pi(0)=\frac12$.

\emph{Singularity transfer.}
    Since $G^{[m]}$ is analytic in a $\Delta$-domain and can be expressed as
    $G^{[m]}(z)=G^{[m]}(\rho)-\frac{K^{[m]}}2(1-4z)^{1/2}+o\bigl((1-4z)^{1/2}\bigr)$ there,
    the transfer theorems~\cite[Theorems~VI.1 and~VI.3]{FS09},
    with $[z^n](1-4z)^{1/2}=-\frac{4^n}{2\sqrt{\pi}\,n^{3/2}}(1+O(1/n))$,
    yield~\eqref{eq:coeff-asympt}.
    Since $|\mathcal{T}_n|\sim4^n/(\sqrt{\pi}\,n^{3/2})$,
    the limit proportion is $\kappa^{[m]}=K^{[m]}/4$.
\end{proof}

\subsection{Bivariate generating function and functional equation}
\label{sec:bivariate-gf}

Rather than solving the hierarchy separately for each $m$, we now encode
all values of $m$ into a single bivariate generating function. This is
the key step that turns the hierarchy into one nonlinear recurrence and
ultimately yields a functional equation for the complete limiting
distribution.

We start by combining the limiting amplitudes $K^{[m]}$ into the
generating function
\[
    K(y) \;=\; \sum_{m\ge1} K^{[m]}\,y^m,
\]
whose coefficients are the quantities that we want to identify.
To understand the behavior of $K(y)$, we introduce, for every
integer $k\ge1$, the bivariate generating function of plane binary
trees of height at most $k$,
\[
    A_k(z,y) \;=\; I_{k-1}(z) + F_k(z,y),
    \qquad
    F_k(z,y) \;=\; \sum_{m\ge1} Q_k^{[m]}(z)\,y^m,
\]
where the variable $y$ marks the pairs of leaves at the deepest level.
Because the marked level coincides with the height of the tree, the
bivariate generating function of plane binary trees graded by their
deepest level is simply
\begin{equation*}
    G(z,y) \;=\; \sum_{m\ge1} G^{[m]}(z)\,y^m
              \;=\; \sum_{k\ge1} F_k(z,y),
\end{equation*}
with $G^{[m]}$ defined by~\eqref{eq:G^[m]-def}.
Note that $G(z,y)$ is a series in $z$ whose coefficients are polynomials
in $y$, whereas $K(y)$ is a series in $y$ alone; the two are related
only through the singular behavior of each $G^{[m]}$ at $z=\rho$, as
established in Theorem~\ref{thm:sqrt-singularity}.

Summing the discrete recurrences~\eqref{eq:Q^[m]-def} over all $m\ge1$
yields a global discrete Riccati equation for $F_k$,
\begin{equation}\label{eq:F-def}
    F_{k+1}(z,y) \;=\; 2z\,I_{k-1}(z)\,F_k(z,y)
                    + z\,\bigl(F_k(z,y)\bigr)^2
    \qquad (k\ge1),
\end{equation}
with initial conditions $F_0(z,y)=0$ and $F_1(z,y)=zy$.
Combinatorially, recurrence~\eqref{eq:F-def} records the same root
decomposition as~\eqref{eq:Q^[m]-def}, while the variable $y$ keeps
track of the total number of pairs of deepest leaves. It is this
compression of the entire hierarchy into a single recurrence that makes
a global description of the distribution possible.

Similarly, adding $I_{k-1}(z)$ to both sides, we obtain the marked
Catalan iteration for~$A_k$,
\begin{equation}\label{eq:A-def}
    A_1(z,y) \;=\; 1+zy,
    \qquad
    A_{k+1}(z,y) \;=\; 1 + z\,A_k(z,y)^2
    \qquad (k\ge1),
\end{equation}
whose initial condition encapsulates the entire dependence on the
deepest level. All the information about the distribution is therefore
carried by the parabolic dynamics of this iteration at the critical
point $z=\rho=1/4$.

The generating function $K(y)$ captures the distribution of leaves at
the deepest level of a~plane binary tree, as shown by the following
theorem. It involves the Fatou coordinate of the parabolic map
$x\mapsto x-x^2/4$, which we first describe. In the variable
$\omega=4/x$, this map becomes
\[
    \mathcal{F}(\omega) \;:=\; \frac{\omega^2}{\omega-1} \;=\; \omega+1+\frac{1}{\omega-1},
    \qquad\text{that is,}\qquad
    \frac{4}{x-x^2/4} = \mathcal{F}\!\left(\frac4x\right),
\]
a perturbation of the translation $\omega\mapsto\omega+1$ near infinity.

\begin{lemma}[Fatou coordinate; see~\cite{milnor}, §10]
\label{lem:fatou}
Let $\mathbb{H}:=\{\omega\in\mathbb{C}:\Re\omega>1\}$ and let
$h(\omega):=\frac{1}{\omega-1}+\log\bigl(1-\frac1\omega\bigr)$ for
$|\omega|>2$.
\begin{enumerate}[\upshape(i)]
    \item For $\omega\in\mathbb{H}$, $\Re \mathcal{F}(\omega)>\Re\omega+1$.
    In particular, $\mathcal{F}(\mathbb{H})\subseteq\{\omega\in\mathbb{C}:\Re\omega>2\}$.
    \item The function
    $\widetilde\Phi(\omega) := \omega-\log\omega+\sum_{n\ge0}h\bigl(\mathcal{F}^n(\omega)\bigr)$
    is analytic on $\{\omega\in\mathbb{C}:\Re\omega>2\}$, can be extended
    analytically to $\mathbb{H}$ and satisfies the Abel equation
    $\widetilde\Phi(\mathcal{F}(\omega))=\widetilde\Phi(\omega)+1$.
    \item For every $\eta\in(0,\frac\pi2)$, there exists a constant $R_\eta>0$ such that
    $\widetilde\Phi$ extends analytically to
    $\mathbb{H}\cup\{\omega\in\mathbb{C}:\, |\omega|>R_\eta,\ |\arg\omega|<\pi-\eta\}$, where
    \[
        \widetilde\Phi(\omega)=\omega-\log\omega+O\bigl(1/|\omega|\bigr)
        \qquad(|\omega|\to\infty).
    \]
\end{enumerate}
Consequently, the function $\Phi(x):=\widetilde\Phi(4/x)$ is analytic on
$\{x\in\mathbb{C}:|x-2|<2\}$, it satisfies
$\Phi\bigl(x-\frac{x^2}{4}\bigr)=\Phi(x)+1$, and
$\Phi(x)=\frac4x+\log x-\log4+O(|x|)$ as $x\to0$ with
$|\arg x|<\pi-\eta$.
\end{lemma}

The elementary proof is given in Appendix~\ref{app:fatou}.

\begin{theorem}[Distribution of leaves at the deepest level]
\label{thm:distribution-leaves}
    The series $K(y)=\sum_{m\ge1}K^{[m]}y^m$ has radius of convergence~$4$,
    and, locally uniformly for $|y|<4$,
    \begin{equation}\label{eq:K(y)-limit}
        K(y) = \lim_{k \to \infty} k^2 F_k(\rho, y)
        = 4\Bigl( \Phi\Bigl(1 - \frac{y}{4}\Bigr) - \Phi(1) \Bigr),
        \qquad \rho=\frac{1}{4}.
    \end{equation}
    Moreover:
    \begin{enumerate}[\upshape(i)]
        \item $K$ satisfies the functional equation
        \begin{equation*}
            K\left( \left(1 + \frac{y}{4}\right)^2 \right) = K(y) + 4,
        \end{equation*}
        and $K(1)=4$;
        \item the coefficients $K^{[m]}$ are related to the limit
        proportions $\kappa^{[m]}$ of trees that have exactly $2m$ leaves at
        the deepest level by $\kappa^{[m]}=K^{[m]}/4$; in particular,
        $\sum_{m\ge1}\kappa^{[m]}=1$;
        \item as $m\to\infty$,
        \[
            K^{[m]} = 4^{-m}\Bigl(16-\frac4m+O\bigl(m^{-2}\bigr)\Bigr),
            \qquad\text{hence}\qquad
            \kappa^{[m]} = 4^{-m+1}\Bigl(1-\frac{1}{4m}+O\bigl(m^{-2}\bigr)\Bigr).
        \]
    \end{enumerate}
\end{theorem}

The identity $\sum_m\kappa^{[m]}=1$ shows that no probability mass
escapes: $(\kappa^{[m]})_{m\ge1}$ is a~genuine probability distribution.
The functional equation also gives the moments of the distribution
(Appendix~\ref{app:moments}): the number $L$ of leaves at the deepest
level satisfies
$\mathbb{E}[L] = 4\kappa^{[1]}$ and
$\operatorname{Var}(L) = 32\kappa^{[2]} + 4\kappa^{[1]} (1-4\kappa^{[1]})$.

\begin{proof}[Proof of Theorem~\ref{thm:distribution-leaves}]
    At the critical point $\rho=1/4$, recurrence~\eqref{eq:A-def} becomes
    \begin{equation*}
        A_1(\rho, y) = 1 + \frac{y}{4}, \qquad
        A_{k+1}(\rho, y) = 1 + \frac{1}{4}A_k(\rho, y)^2
        \quad \mbox{for } k \ge 1,
    \end{equation*}
    and for $y=0$ it reduces to the Catalan iteration~\eqref{eq:L-def},
    since we have $A_k(z,0)=I_{k-1}(z)$ and $I_{k-1}(\rho)\to C(\rho)=2$.
    We therefore set
    \begin{equation*}
        \delta_k(y) := 2 - A_k(\rho, y),
        \qquad\text{so that}\qquad
        \delta_1(y) = 1 - \frac{y}{4}, \qquad
        \delta_{k+1}(y) = \delta_k(y) - \frac{\delta_k(y)^2}{4},
    \end{equation*}
    and, according to~\eqref{eq:A-def} and $A_k(\rho,y)=I_{k-1}(\rho)+F_k(\rho,y)$,
    \begin{equation}\label{eq:F-delta}
        F_k(\rho,y) = \delta_k(0)-\delta_k(y).
    \end{equation}
    The parabolic fixed point $0$ of the map $x\mapsto x-x^2/4$ is sent to infinity by the change of variable $\omega_k(y):=4/\delta_k(y)$, which gives
    \[
        \omega_1(y)=\frac{16}{4-y},
        \qquad
        \omega_{k+1}(y)=\mathcal{F}\bigl(\omega_k(y)\bigr),
        \qquad
        \omega_k(0)=\frac{4}{e_{k-1}(\rho)}.
    \]

\emph{Existence and identification of the limit.}
    Fix $r\in(0,4)$.
    The disc $\{y\in\mathbb{C}:|y|\le r\}$ is mapped by $y\mapsto16/(4-y)$
    onto a closed disc symmetric with respect to the real axis,
    whose real points are $16/(4+r)>2$ and $16/(4-r)$.
    Hence, $\omega_1(y)$ stays in a compact subset of $\{\omega\in\mathbb{C}:\Re\omega>2\}$.
    By Lemma~\ref{lem:fatou}(i), $\Re\omega_k(y)\ge\Re\omega_1(y)+k-1$,
    so $|\omega_j(y)-1|\ge j$ and
    \[
        \omega_k(y) = \omega_1(y)+k-1+\sum_{j<k}\frac{1}{\omega_j(y)-1}
        = k+O(\log k),
    \]
    uniformly for $|y|\le r$.
    The Abel equation gives the relation $\widetilde\Phi(\omega_k(y))=\widetilde\Phi(\omega_1(y))+k-1$,
    and the series of Lemma~\ref{lem:fatou}(ii),
    whose terms satisfy $|h(\omega)|\le2/|\omega|^2$ for $|\omega|\ge2$,
    gives $\widetilde\Phi(\omega_k(y))=\omega_k(y)-\log\omega_k(y)+O(1/k)$.
    Subtracting the same identities for $y=0$ (where $\omega_1(0)=4$), we obtain
    \[
        \omega_k(y)-\omega_k(0)
        = \widetilde\Phi(\omega_1(y))-\widetilde\Phi(4)
          + \log\frac{\omega_k(y)}{\omega_k(0)} + O\Bigl(\frac1k\Bigr)
        = \widetilde\Phi(\omega_1(y))-\widetilde\Phi(4)
          + O\Bigl(\frac{\log k}{k}\Bigr),
    \]
    since $\omega_k(y)/\omega_k(0)=1+O(\log k/k)$.
    Together with~\eqref{eq:F-delta} and $\omega_k(y)\,\omega_k(0)=k^2\bigl(1+O(\log k/k)\bigr)$,
    \[
        k^2F_k(\rho,y)
        = \frac{4k^2\bigl(\omega_k(y)-\omega_k(0)\bigr)}{\omega_k(y)\,\omega_k(0)}
        = 4\bigl(\widetilde\Phi(\omega_1(y))-\widetilde\Phi(4)\bigr)
          + O\Bigl(\frac{\log k}{k}\Bigr),
    \]
    uniformly for $|y|\le r$.
    Each $F_k(\rho,y)$ is a polynomial in $y$,
    so the limit is analytic in $\{y\in\mathbb{C}:|y|<4\}$,
    and by Cauchy's coefficient formula its Taylor coefficients are the limits from Lemma~\ref{lem:marked-scaling}:
    $\lim_k k^2[y^m]F_k(\rho,y)=\lim_k k^2Q_k^{[m]}(\rho)=K^{[m]}$.
    Therefore, the limit is $K(y)$,
    and since $\widetilde\Phi(\omega_1(y))=\Phi(1-y/4)$ and $\widetilde\Phi(4)=\Phi(1)$,
    this proves~\eqref{eq:K(y)-limit}.

\emph{Functional equation.}
    The algebraic identity
    \[
        \left(1 - \frac{y}{4}\right) - \frac{1}{4}\left(1 - \frac{y}{4}\right)^2
        = 1 - \frac{1}{4}\left(1 + \frac{y}{4}\right)^2
    \]
    states that $\delta_2(y)=\delta_1\bigl((1+y/4)^2\bigr)$, that is,
    $\omega_1\bigl((1+y/4)^2\bigr)=\mathcal{F}(\omega_1(y))$.
    For $|y|<4$, we have $|(1+y/4)^2|<4$, and the Abel equation yields
    \[
        K\left(\Bigl(1+\frac y4\Bigr)^2\right)
        = 4\left(\widetilde\Phi\bigl(\mathcal{F}(\omega_1(y))\bigr)-\widetilde\Phi(4)\right)
        = K(y)+4 .
    \]
    For $y=0$, this gives $K(1)=K(0)+4=4$.
    Assertion (ii) is Theorem~\ref{thm:sqrt-singularity}, and then $\sum_m\kappa^{[m]}=K(1)/4=1$.

\emph{Dominant singularity and tail.}
    By~\eqref{eq:K(y)-limit}, $K(y)=4\bigl(\widetilde\Phi(16/(4-y))-\widetilde\Phi(4)\bigr)$ extends analytically to every $y$ such that $16/(4-y)$ lies in the domain of Lemma~\ref{lem:fatou}(iii). First, $\Re\frac{16}{4-y}>1$ if and only if $|y+4|<8$,
    a disc which contains $\{y\in\mathbb{C}:|y|\le4\}\setminus\{4\}$.
    Second, if $0<|y-4|<16/R_\eta$ and $|\arg(y-4)|>\eta$,
    then $\omega=16/(4-y)$ satisfies $|\omega|>R_\eta$ and $|\arg\omega|<\pi-\eta$.
    By compactness of $\{y\in\mathbb{C}:|y|\le4,\ |y-4|\ge16/R_\eta\}$,
    $K(y)$ is therefore analytic in a $\Delta$-domain at $y=4$,
    so that the radius of convergence of $K(y)$ is $4$, and by Lemma~\ref{lem:fatou}(iii),
    \[
        K(y) = 4\omega-4\log\omega-4\widetilde\Phi(4)+O\bigl(1/|\omega|\bigr)
        = \frac{16}{1-y/4} + 4\log\Bigl(1-\frac y4\Bigr) + c_K
          + O\Bigl(\Bigl|1-\frac y4\Bigr|\Bigr)
    \]
    with $c_K=-4\log4-4\widetilde\Phi(4)$.
    The transfer theorems of singularity analysis~\cite[Theorems~VI.1 and~VI.3]{FS09},
    with $[y^m](1-y/4)^{-1}=4^{-m}$ and $[y^m]\log(1-y/4)=-4^{-m}/m$,
    give us finally $K^{[m]}=4^{-m}\bigl(16-4/m+O(m^{-2})\bigr)$, and (iii) follows from (ii).
\end{proof}

\section{Concluding remarks and perspectives}
\label{sec:results-perspectives}

\subsection{Numerical values for the limit distribution}
\label{sec:numerical_values}

To obtain numerical values of the proportions $\kappa^{[m]}$, one has to
evaluate the constants $K^{[m]} = \lim_{k \to \infty} k^2 Q_k^{[m]}(\rho)$.
A direct iteration of the recurrences at $z=\rho$ is not efficient: by
Lemma~\ref{lem:marked-scaling}, we have $k^2Q_k^{[m]}(\rho)=K^{[m]}+O(\log k/k)$,
and the logarithmic term also limits the efficiency of classical
acceleration methods such as Richardson extrapolation~\cite{sidi2003practical}.

The Fatou coordinate provides a much faster scheme. With the notation of
the proof of Theorem~\ref{thm:distribution-leaves}, the Abel equation gives,
for every $k\ge1$,
$K(y)=4\bigl(\widetilde\Phi(\omega_k(y))-\widetilde\Phi(\omega_k(0))\bigr)$,
and Lemma~\ref{lem:fatou}(ii) shows that replacing $\widetilde\Phi(\omega)$ by
$\omega-\log\omega$ in this identity only costs $O(1/k^2)$, uniformly for
$|y|\le1$ (the neglected terms are
$\sum_{n\ge0}\bigl(h(\omega_{k+n}(y))-h(\omega_{k+n}(0))\bigr)$, with
$h'(\omega)=O(|\omega|^{-3})$). We thus iterate
$\delta_{k+1}(y)=\delta_k(y)-\delta_k(y)^2/4$ on power series in $y$ truncated at
a fixed degree, starting from $\delta_1(y)=1-y/4$, and extract the Taylor
coefficients of
$4\bigl[(\omega_k(y)-\log\omega_k(y))-(\omega_k(0)-\log\omega_k(0))\bigr]$
with $\omega_k(y)=4/\delta_k(y)$, in double precision. The values obtained
for $k=10^5$ and $k=10^6$ differ by less than $10^{-9}$; they are also consistent
with the identity $K(1)=\sum_m K^{[m]}=4$ of
Theorem~\ref{thm:distribution-leaves}.

The numerical scheme yields the following values for the limit probability distribution $\kappa^{[m]}$ of finding exactly $2m$ leaves at the deepest level (Figure~\ref{tab:distribution}).


These evaluations offer a geometric perspective on the fringes of large random trees.
\begin{wrapfigure}{r}{0.42\textwidth}
     \vspace{-\baselineskip}
     \centering
         \resizebox{\linewidth}{!}{%
        \begin{tabular}{c c c}
            \hline
            \textbf{Nb. leaves ($2m$)} & \textbf{Constant $K^{[m]}$} & \textbf{Probability $\kappa^{[m]}$} \\
            \hline
            2 & 2.8037194 & 0.7009299 \\
            4 & 0.8902544 & 0.2225636 \\
            6 & 0.2274001 & 0.0568500 \\
            8 & 0.0588010 & 0.0147003 \\
            10 & 0.0148224 & 0.0037056 \\
            \hline
        \end{tabular}
        }
    \caption{Asymptotic probability dis\-tri\-bution of the number of leaves at the deepest level.}
    \label{tab:distribution}
     \vspace{-\baselineskip}
\end{wrapfigure}
The sum of the first five probabilities alone accounts for approximately $99.87\%$ of the distribution, il\-lus\-trating a strong concentration on the minimal configurations.
On the one hand, the first value $\kappa^{[1]} \approx 0.7009$ signifies that, asymptotically, about $70\%$ of plane binary trees possess exactly two leaves at the deepest level.
On the other hand, the asymptotic expectation of the number of leaves at the maximum depth is $\hat{\kappa} = 4\kappa^{[1]} \approx 2.8037$, and its variance is approximately $2.0649$ (Appendix~\ref{app:moments}).

Furthermore, these numerical results empirically validate the analysis conducted in Section~\ref{sec:generalization}.
The successive ratios between consecutive probabilities ($\kappa^{[2]}/\kappa^{[1]} \approx 0.318$, $\kappa^{[3]}/ \kappa^{[2]} \approx 0.255$, $\kappa^{[4]}/ \kappa^{[3]} \approx 0.259$, and $\kappa^{[5]}/ \kappa^{[4]} \approx 0.252$) converge rapidly towards $0.25$, reflecting the exponential decay $\kappa^{[m]} = 4^{-m+1}\bigl(1-\frac1{4m}+O(m^{-2})\bigr)$ of Theorem~\ref{thm:distribution-leaves}; for instance, the two-term approximation gives $0.0037109$ for $\kappa^{[5]}$.
Figure~\ref{fig:kappa-tail} shows that this two-term approximation is
already very accurate for moderate $m$: the geometric decay is visible on
the left, and the correction $1-\frac1{4m}$ appears on the right, where the
normalized probabilities $4^{m-1}\kappa^{[m]}$ approach $1$ from below.
The agreement is in fact much better than the error term $O(m^{-2})$
suggested by Theorem~\ref{thm:distribution-leaves}(iii): for $m\ge12$, the
normalized probabilities and $1-\frac1{4m}$ coincide to six decimal places.
Indeed, the error term is $O(m^{-p})$ for every $p$
(Proposition~\ref{prop:refined-tail} in Appendix~\ref{app:refined-tail}):
beyond $\omega-\log\omega$, the Fatou coordinate only has an asymptotic
expansion in powers of $1/\omega=(4-y)/16$, that is, in polynomials
in~$y$, which do not contribute to the coefficients of large index.

\begin{figure}[t]
\centering
\begin{tikzpicture}
\begin{semilogyaxis}[
    width=0.49\textwidth, height=4.6cm,
    xmin=0.5, xmax=16.5, xtick={1,4,8,12,16},
    ymin=5e-10, ymax=2,
    xlabel={$m$}, ylabel={$\kappa^{[m]}$},
    label style={font=\small}, tick label style={font=\scriptsize},
    axis line style={gray!60}, tick style={gray!60},
    ymajorgrids, grid style={gray!20},
    legend style={font=\scriptsize, draw=none, fill=none, at={(0.97,0.97)}, anchor=north east},
    legend cell align=left,
]
\addplot[domain=1:16, samples=61, thick, gray] {4^(-x+1)*(1-1/(4*x))};
\addlegendentry{$4^{-m+1}\bigl(1-\frac1{4m}\bigr)$}
\addplot[only marks, mark=*, mark size=1.6pt, black] coordinates {(1,7.0092986244e-01) (2,2.2256360258e-01) (3,5.6850031170e-02) (4,1.4700256694e-02) (5,3.7056109124e-03) (6,9.3631801989e-04) (7,2.3539071438e-04) (8,5.9130035509e-05) (9,1.4834700640e-05) (10,3.7193541418e-06) (11,9.3199812286e-07) (12,2.3345163696e-07) (13,5.8458390109e-08) (14,1.4635070171e-08) (15,3.6632020543e-09) (16,9.1677066619e-10)};
\addlegendentry{$\kappa^{[m]}$}
\end{semilogyaxis}
\end{tikzpicture}\hfill
\begin{tikzpicture}
\begin{axis}[
    width=0.49\textwidth, height=4.6cm,
    xmin=0.5, xmax=16.5, xtick={1,4,8,12,16},
    ymin=0.68, ymax=1.02,
    xlabel={$m$}, ylabel={$4^{m-1}\kappa^{[m]}$},
    label style={font=\small}, tick label style={font=\scriptsize},
    axis line style={gray!60}, tick style={gray!60},
    ymajorgrids, grid style={gray!20},
    legend style={font=\scriptsize, draw=none, fill=none, at={(0.97,0.05)}, anchor=south east},
    legend cell align=left,
]
\addplot[domain=0.5:16.5, dashed, gray] {1};
\addlegendentry{$1$}
\addplot[domain=1:16, samples=121, thick, gray] {1-1/(4*x)};
\addlegendentry{$1-\frac1{4m}$}
\addplot[only marks, mark=*, mark size=1.6pt, black] coordinates {(1,0.700930) (2,0.890254) (3,0.909600) (4,0.940816) (5,0.948636) (6,0.958790) (7,0.964160) (8,0.968787) (9,0.972207) (10,0.975006) (11,0.977271) (12,0.979167) (13,0.980769) (14,0.982143) (15,0.983333) (16,0.984375)};
\addlegendentry{$4^{m-1}\kappa^{[m]}$}
\end{axis}
\end{tikzpicture}
\caption{Left: the limiting probabilities $\kappa^{[m]}$ (logarithmic scale)
and their asymptotic approximation from
Theorem~\ref{thm:distribution-leaves}(iii). Right: the normalized
probabilities $4^{m-1}\kappa^{[m]}$, which isolate the correction term
$1-\frac1{4m}$ to the geometric decay.}
\label{fig:kappa-tail}
\end{figure}

\subsection{Discussion and perspectives}
\label{sec:discussion}

As we have seen, the problems studied in this paper exhibit a rigid analytic structure. 
Both the minimal configuration and the complete distribution of the deepest level are governed by the same underlying mechanism: a quasi-logistic recurrence whose critical regime is controlled by a parabolic fixed point.
This leads systematically to a square-root singular expansion at $\rho = 1/4$ and to a universal scaling profile $1/(4\sinh^2(u/2))$ in the critical window.
At a~deeper stage, the entire distribution is encoded by a discrete Riccati cascade that can be linearized through the Abel equation, revealing a macroscopic functional equation and a~universal exponential decay.

The main conclusion of this work is that the described methodology is not specific to binary trees. 
The combination of truncation analysis, discrete integrating factors, and continuous iteration theory forms a toolkit that can be adapted to a wide range of \emph{simple} tree families~\cite{Drmota2009}.
In these families, the same phenomenology is expected to persist: a critical scaling window, a logistic-type dynamics, and a limit distribution governed by a functional equation of Abel type.

However, this apparent universality raises a more fundamental question. 
All these models belong to the class of \emph{simply generated trees}, whose local structure is essentially driven by independent branching mechanisms.
In contrast, entirely different behavior is expected for \emph{increasing trees}, in the sense of~\cite{BFS92}, where labels impose strong global constraints and introduce long-range dependencies in the structure.

This observation leads to the following open problem suggested by our work:
\begin{center}
    \emph{What is the distribution of the deepest level in families of increasing trees?}
\end{center}

In such models (e.g., recursive trees, binary increasing trees, or, more generally, increasing tree families)
the height is typically logarithmic and the profile is known to exhibit fundamen\-tal\-ly different fluctuations.
The quasi-logistic framework developed here no longer applies directly, as the independence structure underlying the Catalan recursion breaks down.
In particular, the truncation error is no longer governed by a simple Riccati equation, and the emergence of a scaling function driven by a local differential equation is far from guaranteed.

Understanding the extreme fringe of increasing trees would, therefore, require new tools, possibly combining analytic combinatorics with probabilistic techniques (such as martingale methods or branching processes with memory).
Whether an analog of the functional equation for the generating function $K(y)$ still exists in this setting, or whether a completely different universality class emerges, remains wide open.

Finally, even within the Catalan framework, a finer question persists. 
Although the functional equation completely characterizes the distribution, the explicit evaluation of the constants~$K^{[m]}$ still relies on numerical methods. 
Finding a direct analytic or combinatorial expression for these constants would provide a deeper understanding of the underlying structure and might serve as a bridge toward more constrained models such as increasing trees.

In summary, while the present work essentially closes the case of simply generated trees with respect to the study of the deepest fringe, it opens a new direction: the study of extreme levels in structured random trees, where global constraints reshape the local geometry.

\newpage

\bibliography{BGN}

\newpage

\appendix

\section{Proofs for Section~\ref{sec:analysis}}
\label{app:analysis}

\begin{proof}[Proof of Proposition~\ref{thm:e-asymp}]
    The proof proceeds by gradually reaching more and more asympto\-tic terms for $e_k(\rho)$.
    First, we show that
    \begin{equation}\label{eq:zero}
        e_k(\rho) \underset{k\to\infty}{=} O\!\left(\frac{1}{k}\right).
    \end{equation}
    Second, we establish the asymptotic relation
    \begin{equation}\label{eq:first}
        e_k(\rho) \underset{k\to\infty}{=} \frac{4}{k} + O\!\left(\frac{\log k}{k^2}\right).
    \end{equation}
    And finally, we obtain the main statement of the proposition.

\emph{Step 1: approximation~\eqref{eq:zero}.}
    At the critical point $z=\rho=1/4$, we have $r(\rho)=1$, and relation~\eqref{eq:e-recur} becomes $e_{k+1}(\rho)=e_k(\rho)-\rho \cdot e_k(\rho)^2$. 
    From the combinatorial meaning, it follows that the sequence $\big(e_k(\rho)\big)_{k\ge0}$ is strictly positive and decreasing to $0$.
    Introducing $v_k(z) = 1/e_k(z)$, we get:
    \begin{equation}\label{eq:v_k(rho)-dif}
        v_{k+1}(\rho) - v_k(\rho)
        = 
        \frac{1}{e_k(\rho)\big(1 - \rho e_k(\rho)\big)} - \frac{1}{e_k(\rho)}
        = 
        \frac{\rho}{1 - \rho e_k(\rho)}.
    \end{equation}
    Since $0 < e_k(\rho) \le 1$, we have $0 < 1 - \rho e_k(\rho) < 1$, which implies that $v_{k+1}(\rho) - v_k(\rho) > \rho$ for all $k \ge 0$.
    Summing the latter inequality from $0$ to $k-1$, we obtain $v_k(\rho) > v_0(\rho) + \rho k$, which unconditionally shows that $e_k(\rho) = O(1/k)$.
    
\emph{Step 2: approximation~\eqref{eq:first}.}
    Asymptotically expanding difference~\eqref{eq:v_k(rho)-dif}, we have:
    \begin{equation}\label{eq:v_k(rho)-dif-expan}
        v_{k+1}(\rho) - v_k(\rho) = \rho + \rho^2 e_k(\rho) + O\big(e_k(\rho)^2\big).
    \end{equation}
    Summing this relation from $0$ to $k-1$, we obtain:
    \[
        v_k(\rho) = v_0(\rho) + \rho k + \rho^2\sum_{i=0}^{k-1} e_i(\rho) + O\left(\sum_{i=0}^{k-1} e_i(\rho)^2\right).
    \]
    We can then use our first estimate~\eqref{eq:zero} to refine the sum via a bootstrapping argument.
    Injecting this bound back, we see that the series $\sum_{i=1}^{\infty} e_i(\rho)^2$ converges, and the partial sum satisfies $\sum_{i=1}^{k-1} e_i(\rho) = O(\log k)$. 
    Thus, $v_k(\rho) = \rho k + O(\log k)$. 
    Inverting the latter relation via $e_k(\rho) = 1/v_k(\rho)$, we obtain the first-order expansion~\eqref{eq:first}.

\emph{Step 3: main approximation.}
    To go one step further, consider again expansion~\eqref{eq:v_k(rho)-dif-expan}.
    Using our second approximation~\eqref{eq:first}, we have
    \begin{equation*}
        v_{k+1}(\rho)-v_k(\rho) = \rho + \frac{\rho}{k} + R_k,
        \qquad R_k = O\!\left(\frac{\log k}{k^2}\right).
    \end{equation*}
    Again, summing the obtained relation from $j=1$ to $k-1$, we get
    \[
        v_k(\rho) = v_1(\rho) + \rho(k-1) + \rho H_{k-1} + \sum_{j=1}^{k-1} R_j,
    \]
    where $H_{k-1} = \log(k-1) + \gamma + O(1/k)$ is the $(k-1)$th harmonic number and $\gamma$ denotes the Euler--Mascheroni constant.

    \medskip
    \noindent
    The series $\sum_{j\ge1} R_j$ converges absolutely, since $R_j = O(\log j/j^2)$. Define
    \[
        S_R := \sum_{j=1}^{\infty} R_j,
        \qquad\text{so that}\qquad
        \sum_{j=1}^{k-1} R_j
        = S_R - \sum_{j=k}^{\infty} R_j.
    \]
    The tail is controlled by the standard integral estimate:
    \begin{equation}\label{eq:tail-bound}
        \left|\sum_{j=k}^{\infty} R_j\right|
        \;\leq\; C\int_{k-1}^{\infty} \frac{\log x}{x^2}\,dx
        \;=\; C\,\frac{\log(k-1)+1}{k-1}
        \;=\; O\!\left(\frac{\log k}{k}\right).
    \end{equation}
    Consequently,
    \begin{equation}\label{eq:v_k-app}
        v_k(\rho) = \rho k + \rho\log k + C_0
        + O\!\left(\frac{\log k}{k}\right),
    \end{equation}
    where the constant
    \begin{equation}\label{eq:C0-def}
        C_0
        := v_1(\rho) - \rho + \rho\gamma + S_R
    \end{equation}
    depends only on the initial value $e_0(\rho)$ (through $v_1(\rho)$ and $S_R$).
    Note that we have retained the error term $O(\log k/k)$ rather than $O(1/k)$: this is the correct remainder, since the tail of the series $\sum R_j$ is genuinely of order $\log k/k$ and not smaller.

    \medskip
    \noindent
    Taking into account that $\rho=1/4$, rewriting relation~\eqref{eq:v_k-app} in the form
    \[
        v_k(\rho) = \frac{k}{4}\left(1 + \frac{\log k + 4C_0}{k}
        + O\!\left(\frac{\log k}{k^2}\right)\right),
    \]
    and using the expansion $(1+u)^{-1} = 1 - u + u^2 + O(u^3)$ for small $u$, we have
    \[
        e_k(\rho) = \frac{1}{v_k(\rho)}
        = \frac{4}{k}
        - \frac{4(\log k + 4C_0)}{k^2}
        + \frac{4(\log k + 4C_0)^2}{k^3}
        + O\!\left(\frac{\log^2 k}{k^3}\right).
    \]
    Note that the term $4(\log k + 4C_0)^2/k^3$ is absorbed into $O(\log^2 k/k^3)$.
    Therefore, we finally obtain
    \[
        e_k(\rho) = \frac{4}{k} - \frac{4\log k}{k^2} + \frac{C}{k^2}
        + O\!\left(\frac{\log^2 k}{k^3}\right),
    \]
    where $C = -16\,C_0$, with $C_0$ defined by~\eqref{eq:C0-def}.
\end{proof}

We prove the two lemmas of Section~\ref{sec:near-parabolic} in the
setting described there (cases (S1) and (S2)). The letters $C,C',\dots$
denote constants depending only on $\theta$ and $U$.

\begin{proof}[Proof of Lemma~\ref{lem:barrier-v}]
    Let us first present the key ideas of the proof.
    Since, by relation~\eqref{eq:e-recur}, $e_{j+1}(z)=e_j(z)\big(1-\tau-ze_j(z)\big)$,
    the inverse $v_j(z)=1/e_j(z)$ satisfies the Riccati recurrence
    \begin{equation}\label{eq:v-recur}
        v_{j+1}(z) = \frac{v_j(z)}{1 - \tau - z/v_j(z)}
        = \alpha v_j(z) \left(1-\frac{\alpha z}{v_j(z)}\right)^{-1},
        \qquad
        \alpha = \frac{1}{1-\tau}.
    \end{equation}
    When $|v_j(z)|$ is large, it is close to the affine recurrence
    \begin{equation}\label{eq:v-linearized}
        \tilde v_{j+1}(z) = \alpha \tilde v_j(z) + \alpha^2 z,
        \qquad
        \tilde v_1(z) := v_1(z),
    \end{equation}
    whose solution is explicit and grows linearly in $j$ (Step~1).
    Let us write $v_j(z)=\tilde v_j(z)+w_j(z)$.
    The perturbation $w_j(z)$ accumulates remainders of size $O(1/|v_i(z)|)$, $i<j$ (Step~2),
    so it is $O(\log j)$ as long as $|v_i(z)|\ge c\,i$ for $i<j$.
    On the other hand, this lower bound on $|v_j(z)|$ follows from the linear growth of $\tilde v_j(z)$ once $w_j(z)$ is known to be small.
    This circularity is resolved by an induction on $j$ (Step~3).

\emph{Step 1: the affine solution.}
    The fixed point of recurrence~\eqref{eq:v-linearized} is
    \[
        v^*(z) := \frac{\alpha^2 z}{1-\alpha}
        = -\frac{z}{(1-\tau)\tau}
        = -\frac{1+\tau}{4\tau},
    \]
    where we used $1-\alpha=-\tau/(1-\tau)$ and $z=(1-\tau)(1+\tau)/4$.
    Hence, for every $j\ge1$,
    \begin{equation*}
        \tilde v_j(z)
        = \alpha^{j-1}\bigl(v_1(z) - v^*(z)\bigr) + v^*(z)
        = \frac{1+\tau}{4\tau}\bigl(\alpha^{j-1}-1\bigr) + \alpha^{j-1}v_1(z).
    \end{equation*}
    Since $\log\alpha=-\log(1-\tau)=\tau+O(|\tau|^2)$,
    for $1\le j\le U/|\tau|$ the value $u_j=j\tau$ differs from the exponent $(j-1)\log\alpha$ by $\tau+(j-1)\,O(|\tau|^2)=O(|\tau|)$.
    Consequently, $\alpha^{j-1}=e^{u_j}+O(|\tau|)$, since $\Re u_j\le U$.
    Moreover, the initial value $v_1(z)$ is bounded, since $v_1(z)=1/e_1(z)\to v_1(\rho)=1/(2-\frac54)=\frac43$ as $z\to\rho$.
    Writing
    \[
        \frac{1+\tau}{4\tau}\bigl(e^{u_j}-1\bigr) =
        \frac{e^{u_j}-1}{4\tau} + \frac{e^{u_j}-1}{4},
    \]
    we obtain
    \begin{equation*}
        \tilde v_j(z) = \frac{e^{u_j}-1}{4\tau} + B_j,
        \qquad |B_j|\le B,
        \qquad 1\le j\le U/|\tau|,
    \end{equation*}
    for some constant $B$.
    Finally, by definition~\eqref{eq:def-eta} of the function $\eta$, we have
    \[
        \frac{e^{u_j}-1}{4\tau} = \frac{j\,\eta(u_j)}{4},
        \qquad u_j\in\mathcal{S}_U,
    \]
    so that $\tilde v_j(z)$ grows linearly:
    \begin{equation}\label{eq:vtilde-barrier}
        \frac{\eta_*}{4}\,j - B \;\le\; |\tilde v_j(z)| \;\le\; \frac{\eta^*}{4}\,j + B.
    \end{equation}

\emph{Step 2: the perturbation.}
    Whenever $|v_j(z)|\ge\frac34$, we have $|\alpha z/v_j(z)|\le\frac12$
    (recall that $|\alpha z|\le\frac38$ in our settings);
    this is the only role of the threshold~$\frac34$.
    For such $j$, expanding the geometric series in recurrence~\eqref{eq:v-recur} to first order gives
    \begin{equation}\label{eq:v-expanded}
        v_{j+1}(z)
        = \alpha v_j(z) + \alpha^2 z + R_j,
        \qquad
        R_j = \frac{\alpha^3 z^2}{v_j(z)}\left(1-\frac{\alpha z}{v_j(z)}\right)^{-1},
        \qquad
        |R_j| \le \frac{C_R}{|v_j(z)|},
    \end{equation}
    with $C_R:=2\sup|\alpha^3z^2|$.
    Subtracting~\eqref{eq:v-linearized}, we can see that the difference $w_j(z) = v_j(z) - \tilde v_j(z)$ satisfies $w_1(z)=0$ and $w_{j+1}(z) = \alpha w_j(z) + R_j$,
    whence, provided $|v_i(z)|\ge\frac34$ for $i<j$,
    \begin{equation}\label{eq:w-representation}
        w_j(z) = \sum_{i=1}^{j-1} \alpha^{j-1-i} R_i.
    \end{equation}
    Since $|\alpha|\le1/(1-|\tau|)\le e^{2|\tau|}$, the powers of $\alpha$ in this sum satisfy $|\alpha|^{j}\le e^{2U}$ for $j\le U/|\tau|$.

\emph{Step 3: bootstrap.}
    Set $c:=\eta_*/8$, which is half the slope in the lower bound of~\eqref{eq:vtilde-barrier};
    the other half will absorb $B$ and the perturbation $w_j(z)$.
    Our goal is to prove by induction on $j\le U/|\tau|$ the property
    \[
        \mathcal P_j : \qquad |v_i(z)| \ge \max\bigl(\tfrac34,\, c\,i\bigr)
        \quad\text{for } 1\le i\le j .
    \]
    The constants below are chosen in the order $c$, $C_w$, $J_0$, $\tau_0$,
    each one depending only on the previous ones (and on $\theta$ and $U$).

\emph{(a) A priori bound on $w_i$.}
    If $\mathcal P_j$ holds, then~\eqref{eq:v-expanded} applied for $i\le j$, with $|R_i|\le C_R/(c\,i)$ and~\eqref{eq:w-representation}, gives, for $i\le j+1$,
    \begin{equation}\label{eq:w-bound}
        |w_i(z)| \;\le\; \frac{e^{2U}C_R}{c}\sum_{\ell=1}^{i-1}\frac1\ell
        \;\le\; C_w\log(i+1),
        \qquad C_w := \frac{2e^{2U}C_R}{c}.
    \end{equation}

\emph{(b) Large indices.}
    Since $\eta_*/4=2c$, the linear term dominates eventually:
    there is $J_0\ge1$ depending only on $c$, $B$ and $C_w$, such that
    \[
        \frac{\eta_*}{4}\,n - B - C_w\log(n+1) \;\ge\; \max\bigl(\tfrac34,\, c\,n\bigr)
        \qquad\text{for all } n\ge J_0 .
    \]

\emph{(c) Small indices.}
    There are finitely many indices $i<J_0$, and we handle them by continuity at $z=\rho$.
    At the critical point, $\tau=0$ and $\alpha=1$,
    so the Riccati recurrence~\eqref{eq:v-recur} reads $v_{i+1}(\rho)=v_i(\rho)/\big(1-\frac{\rho}{v_i(\rho)}\big)\ge v_i(\rho)+\rho$, since $1/(1-x)\ge1+x$ for $0\le x<1$.
    Since $v_0(\rho)=1/\big(C(\rho)-1\big)=1$ and $\rho=\frac{1}{4}$, we get $v_i(\rho)\ge1+\frac i4$.
    For each fixed~$i$, the inverse $v_i(z)=1/e_i(z)$ tends to $v_i(\rho)$ as $z\to\rho$,
    since $e_i(z)=C(z)-I_i(z)$ is continuous and $e_i(\rho)>0$.
    We choose $\tau_0$ so small that $|v_i(z)-v_i(\rho)|\le\frac14$ for all $i<J_0$ and $0<|\tau|<\tau_0$.
    Thus, $|v_i(z)|\ge\frac34+\frac i4\ge\max\big(\frac34,c\,i\big)$ for $i<J_0$, because $c\le\frac18$.
    Hence, $\mathcal P_j$ holds for every $j<J_0$.

\emph{(d) Induction step.}
    If $\mathcal P_j$ holds for some $j\ge J_0-1$ with $j+1\le U/|\tau|$, then, by~\eqref{eq:vtilde-barrier}, \eqref{eq:w-bound} and~(b) with $n=j+1$, we have
    \[
        |v_{j+1}(z)| \ge |\tilde v_{j+1}(z)| - |w_{j+1}(z)|
        \ge \frac{\eta_*}{4}(j+1) - B - C_w\log(j+2)
        \ge \max\bigl(\tfrac34,\, c\,(j+1)\bigr),
    \]
    that is, $\mathcal P_{j+1}$ holds.
    By induction, $\mathcal P_j$ holds for all $j\le U/|\tau|$, and so does~\eqref{eq:w-bound}.

\emph{Conclusion.}
    The lower bound in~\eqref{eq:barrier-v} holds with $c_0:=c$.
    By relations~\eqref{eq:vtilde-barrier} and~\eqref{eq:w-bound}, we have
    $|v_j(z)|\le\frac{\eta^*}{4}j+B+C_w\log(j+1)\le C_0\,j$ for $C_0=\frac{\eta^*}{4}+B+C_w$, and
    \[
        \left|v_j(z) - \frac{e^{u_j}-1}{4\tau}\right|
        \le |B_j|+|w_j(z)| \le B+C_w\log(j+1) \le C_1\log(j+1)
    \]
    with $C_1:=C_w+B$.
\end{proof}

\begin{proof}[Proof of Lemma~\ref{lem:near-parabolic-scaling}]
    The three assertions follow from Lemma~\ref{lem:barrier-v}:
    the first by inverting its estimate on $v_k(z)$,
    the second by inserting the first into an exact expression of $2zI_{k-1}(z)$,
    and the third by specializing the error terms to $k=\Theta(1/|\tau|)$.

\emph{Step 1: the main estimate~\eqref{eq:near-parabolic-main}.}
    Let $1\le k\le U/|\tau|$, and let 
    \[
        a_k(z) := \frac{e^{u_k}-1}{4\tau} = \frac{k\,\eta(u_k)}{4}
    \]
    be the profile of Lemma~\ref{lem:barrier-v}, so that $1/a_k(z)=\tau\phi(u_k)$ and $|a_k(z)|\ge\eta_*k/4$.
    According to Lemma~\ref{lem:barrier-v}, we have $|v_k(z)-a_k(z)|\le C_1\log(k+1)$ and $|v_k(z)|\ge c_0k$.
    Thus, since $e_k(z)=1/v_k(z)$, the error on the inverse is divided by $|a_k|\cdot|v_k|$, which is of
order~$k^2$:
    \[
        \bigl|e_k(z)-\tau\phi(u_k)\bigr|
        = \Bigl|\frac1{v_k(z)}-\frac1{a_k(z)}\Bigr|
        = \frac{|a_k(z)-v_k(z)|}{|a_k(z)|\cdot|v_k(z)|}
        \le \frac{4C_1}{\eta_*c_0}\,\frac{\log(k+1)}{k^2}.
    \]
    This gives us relation~\eqref{eq:near-parabolic-main}.

\emph{Step 2: the factor expansion~\eqref{eq:factor-expansion}.}
    Let $2\le k\le U/|\tau|$, so that Step~1 applies to the $(k-1)$th iteration.
    From the relations $I_{k-1}(z)=C(z)-e_{k-1}(z)$, $2zC(z)=1-\tau$, and $2z=(1-\tau^2)/2$ discussed in Section~\ref{sec:def-not}, we have the exact identity
    \[
        2zI_{k-1}(z) = 1-\tau-\tfrac12\,e_{k-1}(z)+\tfrac12\,\tau^2e_{k-1}(z).
    \]
    We simplify its right-hand side in three steps.
    \begin{itemize}
        \item \emph{The term $\tau^2e_{k-1}(z)$ is $O(|\tau|^2)$.}
            Indeed, since $u_{k-1}=(k-1)\tau$, we have the estimate $|\tau\phi(u_{k-1})|=4/((k-1)|\eta(u_{k-1})|)\le8/(\eta_*k)$, so that, by Step~1, $e_{k-1}(z)$ is bounded.
        \item \emph{Replace $e_{k-1}(z)$ by its profile.}
            According to Step~1 applied to the $(k-1)$th iteration, we have $e_{k-1}(z)=\tau\phi(u_{k-1})+O(\log k/k^2)$.
        \item \emph{Shift $u_{k-1}$ to $u_k$},
            so that the result is expressed in the scaling variable of index~$k$.
            We have $\phi'(u)=-4e^u/(e^u-1)^2=-4e^u/(u\,\eta(u))^2$, hence    $|\phi'(u)|\le4e^U/(\eta_*^2|u|^2)$ on $\mathcal{S}_U$.
            The segment $[u_{k-1},u_k]$, of length $|\tau|$, lies on a ray of $\mathcal{S}_U$ at distance at least $(k-1)|\tau|$ from the origin, so that
            \[
                |\tau|\cdot\bigl|\phi(u_{k-1})-\phi(u_k)\bigr|
                \le |\tau|\cdot|\tau|\cdot\frac{4e^U}{\eta_*^2(k-1)^2|\tau|^2}
                = O\Bigl(\frac1{k^2}\Bigr).
            \]
    \end{itemize}
    Altogether,
    \[
        2zI_{k-1}(z) = 1-\tau\left(1+\frac{\phi(u_k)}{2}\right)
        + O\left(\frac{\log k}{k^2}+|\tau|^2\right),
    \]
    and $1+\phi(u)/2=1+2/(e^u-1)=\coth(u/2)$, which gives us~\eqref{eq:factor-expansion}.

\emph{Step 3: the scaling window.}
    Assume that $u_k$ stays in a compact subset of $\mathcal{S}_U\setminus\{0\}$, so that $r\le|u_k|\le U$ for some $r>0$.
    Then $r/|\tau|\le k\le U/|\tau|$, that is, $k=\Theta(1/|\tau|)$, and
    \[
        \frac{\log k}{k^2}\le\frac{|\tau|^2}{r^2}\log\left(\frac{U}{|\tau|}\right)
        = O\left(|\tau|^2\log\frac1{|\tau|}\right);
    \]
    the other error term $|\tau|^2$ is smaller.
    Finally, notice that $|\tau\phi(u_k)|=4|\tau|/\big(|u_k|\cdot|\eta(u_k)|\big)$ lies between $4|\tau|/(U\eta^*)$ and $4|\tau|/(r\eta_*)$.
    Since the error term in~\eqref{eq:near-parabolic-main} is $o(|\tau|)$, we conclude that $|e_k(z)|=\Theta(|\tau|)$.
\end{proof}

\section{The two elementary problems}
\label{app:two-problems}

\noindent\textbf{Notational conventions.}
Throughout this appendix, we focus on the case $m=1$ of the main
distribution, i.e.\ trees whose deepest level contains exactly two
leaves. To lighten the notation, we write
\[
    \Qone_k(z) \;=\; Q_k^{[1]}(z),
    \qquad
    \Qone(z) \;=\; \sum_{k\ge1}\Qone_k(z) \;=\; G^{[1]}(z),
\]
where $Q_k^{[1]}$ and $G^{[1]}$ are as defined in the main text
(Section~\ref{sec:cascade}). Similarly, the constants $K$ and
$\kappa$ used below are the $m=1$ cases of the constants $K^{[m]}$
and $\kappa^{[m]}$:
\[
    K \;=\; K^{[1]} \;=\; \lim_{k\to\infty} k^2\,\Qone_k(\rho),
    \qquad
    \kappa \;=\; \kappa^{[1]} \;=\; \frac{K^{[1]}}{4}.
\]

\medskip
The complete distribution studied in the main text contains the following two elementary questions as its first probability and its first moment.

\begin{problem}\label{prob:one}
 Evaluate the asymptotic probability that a plane binary tree of size $n$ (that is, with $n$ internal nodes) has exactly two leaves at the deepest level.
\end{problem}

\begin{problem}\label{prob:two}
 Compute the asymptotic mean of the total number of leaves at the deepest level of a plane binary tree of size $n$.
\end{problem}

\subsection{Generating functions for the two elementary problems}
To encode Problem~\ref{prob:one}, for any fixed $k\in\mathbb{Z}_{\ge0}$, we consider the generating function $\Qone_k(z)$ of plane binary trees of height $k$ with exactly two leaves at the $k$th level.
The sequence of polynomials $\Qone_k(z)$ satisfies the recurrence
\begin{equation}\label{eq:Q1-def}
    \Qone_0(z)=0, \qquad 
    \Qone_1(z)=z, \qquad 
    \Qone_{k+1}(z) = 2zI_{k-1}(z)\Qone_k(z) \quad 
    \mbox{for } k \ge 1.
\end{equation}
As a particular case of recurrence~\eqref{eq:Q^[m]-def}, recurrence~\eqref{eq:Q1-def} is interpreted combinatorially in the same manner:
to obtain a~tree of height $k+1$ with two leaves at the $(k+1)$th level, we attach two subtrees to the root.
One of these subtrees is of height $k$ with exactly two leaves at the $k$th level (i.e., it is encoded by~$\Qone_k(z)$), and the other subtree is of height at most $k-1$ (so it is encoded by~$I_{k-1}(z)$).
The factor $2$ refers to two possible orders in which the subtrees can be attached to the root.

Problem~\ref{prob:two} is derived from an additive marking parameter: we additionally mark a leaf at the deepest level of a tree.
Formally, given an integer $k\in\mathbb{Z}_{\ge0}$, we denote by $\Qtwo_{k}(z)$ the generating function of plane binary trees of height $k$ with a distinguished leaf at the $k$th level.
Compared to $\Qone_{k}(z)$, the polynomials $\Qtwo_{k}(z)$ satisfy a shifted recurrence
\begin{equation*}
    \Qtwo_0(z)=1, \qquad
    \Qtwo_{k+1}(z) = 2zI_{k}(z)\Qtwo_k(z) \quad
    \mbox{for } k \ge 0,
\end{equation*}
which can be explained combinatorially in the same way as~\eqref{eq:Q1-def}.

The quantities sought in Problems~\ref{prob:one} and~\ref{prob:two} are then extracted from the two global generating series
\[
    \Qone(z)=\sum\limits_{k\ge0}\Qone_k(z)
    \qquad \text{and} \qquad
    \Qtwo(z)=\sum\limits_{k\ge0}\Qtwo_k(z).
\]

\begin{lemma}\label{lem:Q1=zQ2}
    The following identities hold:
    \[
        \Qone_{k+1}(z) = z\Qtwo_k(z)
        \quad\mbox{for } k\ge0
        \qquad \text{and} \qquad
        \Qone(z) = z\Qtwo(z).
    \]
\end{lemma}
\begin{proof}
    The proof follows from the fact that having two leaves at the deepest level of a tree of size $n$ is equivalent to choosing a single leaf at the deepest level of a tree of size $n-1$ and replacing it with a node with two leaves.
\end{proof}

Lemma~\ref{lem:Q1=zQ2} shows that the two series $\Qone(z)$ and $\Qtwo(z)$ share the same analytical backbone.
We therefore focus on Problem~\ref{prob:one} throughout this appendix.

\subsection{Asymptotic analysis of the two-leaf case}
\label{sec:Q-analysis}

This subsection collects the estimates on $\Qone_k(z)$ used in
Section~\ref{sec:generalization}, including the proof of the comparison
Lemma~\ref{lem:Q-comparison}.
Throughout the rest of this appendix, we use the product
form~\eqref{eq:prod_Qk} of recurrence~\eqref{eq:Q1-def}, and the notation of Section~\ref{sec:near-parabolic}, in particular the
settings (S1)--(S2), the compact set $\mathcal{S}_U$ and the positive real value $\tau_0$ determined by Lemma~\ref{lem:barrier-v}.

\subsection{Refined critical expansion of \texorpdfstring{$\Qone_k(\rho)$}{Q1\_k(rho)}}
\label{app:Qk-rho}

\begin{lemma}\label{lem:Qk-rho}
    As $k\to\infty$, the critical values $\Qone_k(\rho)$ satisfy
    \[
        \Qone_k(\rho) = \frac{K}{k^{2}}+O\!\left(\frac{\log k}{k^{3}}\right),
        \qquad\text{where}\quad K=K^{[1]}=\lim_{k\to\infty}k^{2}\Qone_k(\rho)>0 .
    \]
\end{lemma}

\begin{proof}
    We start by substituting $z=\rho$ in relation~\eqref{eq:prod_Qk}, which gives us
    \begin{equation}\label{eq:prod_Qk_bis}
        Q_k(\rho) = \rho \prod_{\ell=0}^{k-2} 2\rho I_{\ell}(\rho),
        \qquad k\ge2.
    \end{equation}
    It is clear from the combinatorial meaning of the polynomials $I_{\ell}(z)$ that all the factors $2\rho I_{\ell}(\rho)$ are positive.
    Identities $C(\rho)=2$ and $2\rho=\frac12$ with relation~\eqref{eq:e-def} imply $2\rho I_{\ell}(\rho)=1-\frac12 e_{\ell}(\rho)$.
    For $\ell=0$, we have $2\rho I_{0}(\rho)=2\rho=\frac12$, which contributes $\log\frac12$ to the logarithm of the product in the right-hand side of~\eqref{eq:prod_Qk_bis}.
    For $\ell\ge1$, it follows from Proposition~\ref{thm:e-asymp} that
    \[
        1-\tfrac12 e_{\ell}(\rho)
        =1-\frac{2}{\ell}+O\!\left(\frac{\log\ell}{\ell^2}\right).
    \]
    Taking the logarithm and using $\log(1+x)=x-\frac{x^{2}}{2}+O(x^{3})$, we get
    \[
        \log\bigl(2\rho I_{\ell}(\rho)\bigr)
        =-\frac{2}{\ell}+O\!\left(\frac{\log\ell}{\ell^2}\right), \qquad \ell\ge1 .
    \]
    Summing from $\ell=1$ to $k-2$ and adding the term corresponding to $\ell=0$ gives
    \[
        \log \Qone_k(\rho) = \log\rho + \log\frac12 +
        \sum_{\ell=1}^{k-2}\left(-\frac{2}{\ell}+O\!\left(\frac{\log\ell}{\ell^2}\right)\right).
    \]
    The harmonic sum is
    \[
        -2\sum_{\ell=1}^{k-2}\frac{1}{\ell} = -2\log k - 2\gamma +O\!\left(\frac{1}{k}\right).
    \]
    The error terms $O(\log\ell/\ell^2)$ sum to a convergent series with the tail
    \[
        \sum_{\ell\ge k-1}O\!\left(\frac{\log\ell}{\ell^2}\right) = O\!\left(\frac{\log k}{k}\right),
    \]
    similarly to~\eqref{eq:tail-bound}. 
    Collecting all constant contributions into a single constant $L$,
    we obtain $\log \Qone_k(\rho)=-2\log k+L+O(\log k/k)$, and exponentiating,
    \[
        \Qone_k(\rho)=\frac{e^{L}}{k^{2}}\exp\!\left(O\!\left(\frac{\log k}{k}\right)\right)
        =\frac{K}{k^{2}}+O\!\left(\frac{\log k}{k^{3}}\right),
        \qquad K=e^{L}>0. \qedhere
    \]
\end{proof}

\subsection{Proof of Lemma~\ref{lem:Q-comparison} and consequences}
\label{sec:Q-comparison}

The comparison Lemma~\ref{lem:Q-comparison} is stated in
Section~\ref{sec:cascade}, with the functions $\psi$ and $\vartheta$
defined by~\eqref{eq:psi-def}. Here, we prove it and derive its consequences.

\begin{proof}[Proof of Lemma~\ref{lem:Q-comparison}]
    Let us first introduce the proof strategy. According to relation~\eqref{eq:prod_Qk},
    the ratio $Q_k^{[1]}(z)/Q_k^{[1]}(\rho)$ is a product of ratios of factors,
    so we compare its logarithm with $\log\psi(u_k)=-\int_0^{u_k}\vartheta(s)\,ds$.
    Each factor ratio is $1-\tau\vartheta(u_j)$ up to small errors,
    because the singular parts $2/j$ of the two factors cancel out (Step~1).
    The errors are summable, but they are small only for large $j$ (Step~2);
    the main terms form a Riemann sum of $\vartheta$ (Step~3).
    The finitely many factors with small $j$ are handled by continuity at $z=\rho$ (Step~4).

\emph{Step 1: comparison of the factors.}
    Let $a_j=a_j(z):=2zI_{j-1}(z)$ and $b_j:=2\rho I_{j-1}(\rho)$, so that
    \[
        \frac{Q_k^{[1]}(z)}{Q_k^{[1]}(\rho)}=\frac z\rho\prod_{j=1}^{k-1}\frac{a_j}{b_j}
    \]
    due to product form~\eqref{eq:prod_Qk}.
    At the critical point, since $0<e_{j-1}(\rho)\le e_0(\rho)=1$,
    we have $b_j=1-\frac12e_{j-1}(\rho)\in[\frac12,1)$ and,
    by Proposition~\ref{thm:e-asymp}, $b_j=1-2/j+s_j$ with $|s_j|\le C\log j/j^2$.
    Near the critical point, according to~\eqref{eq:factor-expansion},
    for $2\le j\le U/|\tau|$ we have
    $a_j=1-\tau\coth(u_j/2)+r_j$ with $|r_j|\le C(\log j/j^2+|\tau|^2)$.
    Using the relation $\tau\coth(u_j/2)=2/j+\tau\vartheta(u_j)$ that follows from~\eqref{eq:psi-def},
    as well as the bound $|\vartheta(u)|\le C|u|$ on $\mathcal{S}_U$, we obtain
    \begin{equation}\label{eq:aj-bj}
        a_j-b_j = -\tau\vartheta(u_j)+r_j-s_j,
        \qquad
        |a_j-b_j| \le C\left(j|\tau|^2+\frac{\log j}{j^2}+|\tau|^2\right).
    \end{equation}

\emph{Step 2: logarithms of the factors.}
    The bound~\eqref{eq:aj-bj} is small only when $j$ is large,
    so we split the product at an index $J$.
    Fix $J\ge2$ large enough, so that $C\log j/j^2\le\frac18$ for $j>J$.
    Then, for $|\tau|$ small enough, $|a_j-b_j|\le\frac14$;
    hence, $|a_j/b_j-1|\le\frac12$ for all $J<j\le U/|\tau|$.
    For $j\le J$, each $a_j=a_j(z)$ is a polynomial in $z$ with $a_j(\rho)=b_j$,
    and so $|a_j/b_j-1|\le\frac12$ for $z$ close enough to $\rho$.
    Thus, we may define
    \[
        \log\frac{Q_k^{[1]}(z)}{Q_k^{[1]}(\rho)}
        := \log\frac z\rho+\sum_{j=1}^{k-1}\log\frac{a_j}{b_j}
    \]
    with principal logarithms. 
    Set $x_j:=(a_j-b_j)/b_j$, so that $a_j/b_j=1+x_j$.
    For $j>J$, we linearize logarithms in two moves:
    \[
        \left|\log\frac{a_j}{b_j}-(a_j-b_j)\right|\le
        \left|\log\frac{a_j}{b_j}-x_j\right| + \bigl|x_j-(a_j-b_j)\bigr|.
    \]
    For the first summand in the right-hand side, we have $|\log(1+x_j)-x_j|\le|x_j|^2$.
    As for the second summand, it satisfies $|x_j-(a_j-b_j)|=|a_j-b_j|\,|1/b_j-1|\le C|a_j-b_j|/j$,
    because $1/b_j-1\le2(1-b_j)=e_{j-1}(\rho)\le C/j$.
    Together with~\eqref{eq:aj-bj}, this gives
    \[
        \log\frac{a_j}{b_j} = -\tau\vartheta(u_j)+\varepsilon_j,
        \qquad
        |\varepsilon_j| \le |r_j|+|s_j|+\frac{C|a_j-b_j|}{j}+4|a_j-b_j|^2 .
    \]
    We sum these errors over $J<j<k\le U/|\tau|$, using~\eqref{eq:aj-bj} and $j|\tau|\le U$:
    the terms $|r_j|+|s_j|$ contribute $O\big(\sum_{j>J}\log j/j^2\big)+O(|\tau|)$,
    the terms $|a_j-b_j|/j$ contribute $O\big(\sum_{j>J}\log j/j^3\big)+O\big(|\tau|\log\frac1{|\tau|}\big)$,
    and the terms $|a_j-b_j|^2$ contribute $O\big(\sum_{j>J}\log^2j/j^4\big)+O(|\tau|)$.
    Hence
    \[
        \sum_{J<j<k}\log\frac{a_j}{b_j}
        = -\tau\sum_{J<j<k}\vartheta(u_j)
        + O\Bigl(\sum_{j>J}\frac{\log j}{j^2}\Bigr)
        + O\left(|\tau|\log\frac1{|\tau|}\right).
    \]

\emph{Step 3: the Riemann sum.}
    Since $\vartheta(0)=0$ and $\vartheta'$ is bounded on $\mathcal{S}_U$,
    the sum $\tau\sum_{j=1}^{k-1}\vartheta(u_j)$ is a Riemann sum along the segment $[0,u_k]$,
    with at most $U/|\tau|$ steps of length $|\tau|$,
    so that $\bigl|\tau\sum_{j=1}^{k-1}\vartheta(u_j)-\int_0^{u_k}\vartheta(s)\,ds\bigr|\le C|\tau|$.
    The terms $j\le J$ missing in the Riemann sum are small: $|\tau\sum_{j\le J}\vartheta(u_j)|\le CJ^2|\tau|^2$.
    Due to~\eqref{eq:int-theta}, $\int_0^{u_k}\vartheta(s)\,ds=-\log\psi(u_k)$.

\emph{Step 4: conclusion.}
    Combining Steps~2 and~3, we obtain, uniformly for $1\le k\le U/|\tau|$,
    \[
        \left|\log\frac{Q_k^{[1]}(z)}{Q_k^{[1]}(\rho)\,\psi(u_k)}\right|
        \le |\log(z/\rho)| + \sum_{j=1}^{J}\left|\log\frac{a_j}{b_j}\right|
          + C\sum_{j=J+1}^{k-1}\frac{\log j}{j^2} + C_J\,|\tau|\log\frac1{|\tau|}
    \]
    (for $k\le J+1$, the sum over $J<j<k$ is empty).
    The third term is the tail of a convergent series and does not depend on $z$;
    the other three tend to $0$ as $z\to\rho$ for fixed $J$, since $a_j(\rho)=b_j$.
    Given $\epsilon'>0$, we first choose $J$ so that the third term is at most $\epsilon'/2$,
    and then we choose $z$ close enough to $\rho$ so that the sum of the other terms is at most $\epsilon'/2$.
    This proves the lemma.
\end{proof}

\begin{corollary}\label{cor:Q-bounds}
In both settings (S1) and (S2), there exist constants $A_U,\tau_1>0$ such that,
for $0<|\tau|<\tau_1$:
\begin{enumerate}[\upshape(i)]
    \item $\Qone_k(z)\neq0$ and $|\Qone_k(z)|\le A_U/k^2$ for
    $1\le k\le U/|\tau|$;
    \item $|I_{j-1}(z)|\ge\frac12$ for $1\le j\le U/|\tau|$;
    \item if $u_k$ stays in a compact subset of
    $\mathcal{S}_U\setminus\{0\}$, then, as $z\to\rho$,
    \[
        \Qone_k(z) = K\tau^2 g(u_k)\bigl(1+o(1)\bigr),
        \qquad
        \Qone_k(\rho)-\Qone_k(z) = K\tau^2 \varpi(u_k)+o(\tau^2),
    \]
    where\quad $g(u):=\dfrac{e^u}{(e^u-1)^2}=\dfrac{1}{4\sinh^2(u/2)}$\quad and\quad
    $\varpi(u):=\dfrac{1}{u^2}-g(u)$.
\end{enumerate}
\end{corollary}

\begin{proof}
    (i) follows from Lemmas~\ref{lem:Q-comparison} and~\ref{lem:Qk-rho},
    since $\psi(u)\neq0$ and $|\psi(u)|\le e^U/\eta_*^2$ on $\mathcal{S}_U$.\\
    (ii) For $j=1$, we have $I_0(z)=1$.
    For $j\ge2$, we have $I_{j-1}(z)=C(z)-e_{j-1}(z)$ with $C(z)=2/(1+\tau)$,
    and $|e_{j-1}(z)|\le1/\big(c_0(j-1)\big)$ by Lemma~\ref{lem:barrier-v}.
    Therefore, $|I_{j-1}(z)|\ge\frac12$ for $j\ge J_1:=2+\lceil2/c_0\rceil$ and $|\tau|$ small.
    For $j<J_1$, we have $I_{j-1}(\rho)=2-e_{j-1}(\rho)\ge1$, and $I_{j-1}(z)$ is a polynomial,
    so $|I_{j-1}(z)|\ge\frac12$ for $z$ close to $\rho$.\\
    (iii) Here $k\ge r/|\tau|\to\infty$ for some $r>0$,
    so Lemma~\ref{lem:Qk-rho} gives $\Qone_k(\rho)=K\tau^2u_k^{-2}\big(1+o(1)\big)$, which is $O(\tau^2)$.
    The first estimate follows from Lemma~\ref{lem:Q-comparison} since $\psi(u)/u^2=g(u)$,
    and the second one comes from
    $\Qone_k(\rho)-\Qone_k(z)=\Qone_k(\rho)\bigl(1-\psi(u_k)\bigr)+\Qone_k(\rho)\psi(u_k)\,o(1)$,
    since $\big(1-\psi(u)\big)/u^2=\varpi(u)$ is bounded on $\mathcal{S}_U$.
\end{proof}

\subsection{Exponential tail}
\label{sec:Q-tail}

Here we will use the fact that, in setting (S1), $\Re\tau\ge\sigma|\tau|$ with $\sigma:=\sin(\theta/2)>0$.

\begin{lemma}\label{lem:Q-tail}
In setting (S1), there exist constants $U_0\ge1$, $A,c,\tau_2>0$ such that,
for $0<|\tau|<\tau_2$ and every $k\ge U_0/|\tau|$,
\[
    |e_k(z)|\le\frac{\sigma|\tau|}{4},\qquad
    |2zI_k(z)|\le e^{-c|\tau|},\qquad
    |I_k(z)|\ge\frac12,\qquad
    |\Qone_k(z)|\le A|\tau|^2e^{-ck|\tau|}.
\]
Consequently, there exists $A'>0$ such that $|\Qone_k(z)|\le A'/k^2$ for
all $k\ge1$ and $0<|\tau|<\tau_2$.
\end{lemma}

\begin{proof}
    We begin by explaining how the proof proceeds.
    Beyond the scaling window, that is, for $k\ge U_0/|\tau|$ with $U_0$ large,
    the truncation error $e_k(z)$ is small compared with $|\tau|$.
    The recurrence~\eqref{eq:e-recur} written in the form $e_{k+1}(z)=e_k(z)\big(1-\tau-ze_k(z)\big)$
    is then dominated by its linear part, whose coefficient satisfies $|1-\tau|<1$ in setting (S1):
    $e_k(z)$ remains small (Step~2), and the factors $2zI_k(z)$ become uniformly contracting (Step~3).
    Consequently, $\Qone_k(z)$ decays geometrically from its value at the entrance of this regime,
    which is $O(|\tau|^2)$ by Corollary~\ref{cor:Q-bounds} (Step~4).
    The constant $U_0$ depends only on~$\sigma$,
    and $\tau_2$ is chosen small enough for all the estimates below to hold.

\emph{Step 0: the linear coefficient.}
    For $|\tau|\le\sigma$, since $\Re\tau\ge\sigma|\tau|$, we have
    \[
        |1-\tau|^2=1-2\Re\tau+|\tau|^2\le1-2\sigma|\tau|+|\tau|^2
        \le\left(1-\frac{\sigma|\tau|}{2}\right)^2 .
    \]

\emph{Step 1: entrance into the regime.}
    Choose $U_0\ge1/\sigma$ such that $17U_0^2e^{-\sigma U_0}\le\sigma/4$,
    and set $k_*:=\lceil U_0/|\tau|\rceil\le2U_0/|\tau|$, the first index of the regime.
    Then $u_{k_*}$ lies in the compact set $\mathcal{S}_{2U_0}\cap\{|u|\ge U_0\}$,
    and Lemma~\ref{lem:near-parabolic-scaling} (with $U=2U_0$) gives $e_{k_*}(z)=\tau\phi(u_{k_*})+O(|\tau|^2\log\frac1{|\tau|})$.
    As $\Re u_{k_*}\ge\sigma U_0\ge1$, we have $|e^{u_{k_*}}-1|\ge e^{\Re u_{k_*}}/2$. 
    Therefore, the profile $\phi$ is exponentially small there:
    \[
        |\phi(u_{k_*})| = \left|\frac{u_{k_*}^2e^{u_{k_*}}}{(e^u_{k_*}-1)^2}\right| \le
        \frac{(2U_0)^2e^{\Re u_{k_*}}}{(e^{\Re u_{k_*}}/2)^2} \le 16U_0^2e^{-\sigma U_0}    
    \]
    It follows from this inequality that, for $|\tau|$ small,
    \[
        |e_{k_*}(z)|\le17U_0^2|\tau|e^{-\sigma U_0}\le\frac{\sigma|\tau|}{4}.
    \]

\emph{Step 2: the error stays small.}
    If $|e_k(z)|\le\sigma|\tau|/4$ for some $k\ge k_*$, then,
    using Step~0 and taking into account that $|z|\le\frac12$, we have
    \[
        \bigl|e_{k+1}(z)\bigr| = \bigl|e_k(z)\bigr| \cdot \bigl|1-\tau-ze_k(z)\bigr|
        \le \bigl|e_k(z)\bigr|\left(1-\frac{\sigma|\tau|}{2}+\frac{\sigma|\tau|}{8}\right)
        \le \bigl|e_k(z)\bigr| .
    \]
    By induction, $|e_k(z)|\le\sigma|\tau|/4$ for all $k\ge k_*$, which is the first bound of the lemma.

\emph{Step 3: the factors.}
    For $k\ge k_*$, since $2zC(z)=1-\tau$ and $|2z|\le1$, Steps~0 and~2 together with~\eqref{eq:e-def} give
    \[
        \big|2zI_k(z)\big| = \big|1-\tau-2ze_k(z)\big| \le
        1-\frac{\sigma|\tau|}{2}+\frac{\sigma|\tau|}{4}
        =1-\frac{\sigma|\tau|}{4}\le e^{-\sigma|\tau|/4},
    \]
    which is the second bound with $c:=\sigma/4$.
    Moreover, $C(z)=2/(1+\tau)$, and therefore
    $|I_k(z)|\ge|C(z)|-|e_k(z)|=2/|1+\tau|-|e_k(z)|\ge\frac12$.

\emph{Step 4: geometric decay of $\Qone_k(z)$.}
    Let $k>k_*$.
    According to~\eqref{eq:prod_Qk},
    \[
        \Qone_k(z)=\Qone_{k_*+1}(z)\prod_{j=k_*+1}^{k-1}2zI_{j-1}(z),
    \]
    where all the $k-k_*-1$ factors have indices $j-1\ge k_*$.
    Hence, the modulus of each factor is at most $e^{-c|\tau|}$ by Step~3.
    The initial value is controlled by the scaling window: since $k_*+1\le(2U_0+1)/|\tau|$,
    Corollary~\ref{cor:Q-bounds}(i) (with $U=2U_0+1$) gives us the inequality
    $|\Qone_{k_*+1}(z)|\le A_{2U_0+1}/(k_*+1)^2\le A_{2U_0+1}|\tau|^2/U_0^2$.
    Therefore
    \[
        \big|\Qone_k(z)\big|\le\frac{A_{2U_0+1}}{U_0^2}\,|\tau|^2e^{-c|\tau|(k-k_*-1)}
        \le A\,|\tau|^2e^{-ck|\tau|},
        \qquad A:=\frac{A_{2U_0+1}}{U_0^2}\,e^{c(2U_0+1)},
    \]
    since $(k_*+1)|\tau|\le2U_0+1$.
    The case $k=k_*$ follows in the same way from Corollary~\ref{cor:Q-bounds}(i).

\emph{Step 5: the uniform bound.}
    For $k\ge U_0/|\tau|$, Step~4 gives
    \[
        \big|\Qone_k(z)\big|\le\frac{A}{k^2}\,\big(k|\tau|\big)^2e^{-ck|\tau|}
        \le\frac{A}{k^2}\,\sup_{x\ge0}x^2e^{-cx}
        =\frac{A}{k^2}\Bigl(\frac{2}{ce}\Bigr)^2 ,
    \]
    while for $k<U_0/|\tau|$, Corollary~\ref{cor:Q-bounds}(i) (with $U=U_0$) gives $|\Qone_k(z)|\le A_{U_0}/k^2$.
    This proves the last assertion with $A':=\max\bigl(A_{U_0},\,4A/(ce)^2\bigr)$.
\end{proof}

\subsection{Answers to the two elementary problems}

Problem~\ref{prob:one} is the case $m=1$ of
Theorem~\ref{thm:sqrt-singularity}: the limit proportion of trees with
exactly two leaves at the deepest level is
$\kappa=\kappa^{[1]}=K^{[1]}/4\approx0.7009$. For
Problem~\ref{prob:two}, the total number of leaves at the deepest level,
summed over all trees of size $n$, is $[z^n]\Qtwo(z)=[z^{n+1}]\Qone(z)$
by Lemma~\ref{lem:Q1=zQ2}. By~\eqref{eq:coeff-asympt} with $m=1$ and
$|\mathcal{T}_n|\sim4^n/(\sqrt\pi\,n^{3/2})$, the mean number of leaves at
the deepest level tends to
\[
    \lim_{n\to\infty}\frac{[z^{n+1}]\Qone(z)}{|\mathcal{T}_n|}
    = 4\cdot\frac{K^{[1]}}{4} = K^{[1]} = 4\kappa^{[1]}\approx2.8037,
\]
in accordance with the value $\mathbb{E}[L]=4\kappa^{[1]}$ obtained from
the functional equation in Appendix~\ref{app:moments}.

\section{Proofs for Section~\ref{sec:generalization}}
\label{app:proofs-sec3}

This appendix contains the auxiliary lemmas used in the proof of
Theorem~\ref{thm:sqrt-singularity} and the proof of
Lemma~\ref{lem:fatou}. Throughout, $\Delta=\Delta_\varepsilon$ is a sufficiently
small $\Delta$-domain at $\rho$ (see setting (S1) of
Section~\ref{sec:near-parabolic}), so that the results of
Appendix~\ref{app:two-problems} apply with the constants $U_0,c$ of
Lemma~\ref{lem:Q-tail}; we write $\bar\Delta:=\Delta\cup\{\rho\}$.

\subsection{Small heights}

\begin{lemma}\label{lem:small-zone}
For every $m\ge1$, there exists $M_m>0$ such that, for every $k\ge1$ and
every $z\in\Delta$ with $k|\tau|\le1$,
\[
    \left|Q_k^{[m]}(z)-Q_k^{[m]}(\rho)\right|\le M_m|\tau|^2 .
\]
\end{lemma}

\begin{proof}
    Let us first describe how we derive the result.
    Since $|z-\rho|=|\tau|^2/4$,
    it suffices to bound the derivative of the polynomial $Q_k^{[m]}(z)$ near $\rho$, uniformly in $k$.
    We obtain the bound from Cauchy's estimates, on a disc of radius of order $1/k^2$ around $\rho$,
    which is exactly the region where $k|\tau|$ stays bounded.
    This requires a bound $|Q_k^{[m]}(z)|\le B_m/k^2$ on a full neighborhood of $\rho$, and not only in $\Delta$ (Step~1).
    This is why we use setting (S2), in which the variable $\xi$ ranges over a punctured disc.

\emph{Step 1: a bound on a punctured disc.}
    We apply the results of Section~\ref{sec:near-parabolic} and Appendix~\ref{app:two-problems} in
setting (S2),
    on a punctured disc $\{\xi\in\mathbb{C}:0<|\xi-\rho|<\varepsilon'\}$,
    with $U=6$ and $\varepsilon'>0$ sufficiently small.
    By Corollary~\ref{cor:Q-bounds}, there is $\tau_1>0$ such that,
    whenever $0<|\tau(\xi)|<\tau_1$ and $1\le j\le6/|\tau(\xi)|$,
    \[
        \Qone_j(\xi)\ne0,
        \qquad
        \bigl|\Qone_j(\xi)\bigr|\le\frac{A_6}{j^2},
        \qquad
        \bigl|I_{j-1}(\xi)\bigr|\ge\frac12 .
    \]
    Note that these bounds also hold at $\xi=\rho$
    (with $|\Qone_j(\rho)|\le A_6/j^2$ after possibly enlarging~$A_6$).
    We now argue as in the proof of Lemma~\ref{lem:marked-scaling} (uniform bounds on $W_k^{[m]}$),
    with $T$ replaced by $T':=2A_6\ge A_6\sum_{j\ge1}j^{-2}$.
    Set $\bar\omega_1':=1$ and $\bar\omega_m':=T'\sum_{i=1}^{m-1}\bar\omega_i'\bar\omega_{m-i}'$.
    Since the increments in~\eqref{eq:duhamel} only involve indices $j<k$,
    the same induction on $m$ gives $|W_k^{[m]}(\xi)|\le\bar\omega_m'$, hence
    \[
        |Q_k^{[m]}(\xi)|\le\frac{B_m}{k^2},
        \qquad B_m:=\bar\omega_m'A_6,
    \]
    under the same conditions on $\xi$ and $k$.

\emph{Step 2: the bound holds on a disc of radius $2r_k$.}
    Let $r_k:=1/(4k^2)$ and choose $k_1$ such that $2r_{k_1}<\varepsilon'$ and $\sqrt2/k_1<\tau_1$.
    For $k\ge k_1$ and $|\xi-\rho|\le2r_k$, we have $|\tau(\xi)|^2=4|\xi-\rho|\le2/k^2$, so that
    \[
        |\tau(\xi)|<\tau_1
        \qquad\text{and}\qquad
        k\le\frac{\sqrt2}{|\tau(\xi)|}\le\frac{6}{|\tau(\xi)|} .
    \]
    By Step~1, $|Q_k^{[m]}(\xi)|\le B_m/k^2$ on the closed disc $\bar D(\rho,2r_k)$.

\emph{Step 3: Cauchy's estimates.}
    For every point of $\bar D(\rho,r_k)$,
    the circle of radius $r_k$ centered at this point lies in $\bar D(\rho,2r_k)$.
    Cauchy's estimate for the polynomial $Q_k^{[m]}(\xi)$ therefore gives
    \[
        \left|\bigl(Q_k^{[m]}\bigr)'(\xi)\right|\le\frac{B_m/k^2}{r_k}=4B_m,
        \qquad\xi\in\bar D(\rho,r_k).
    \]

\emph{Conclusion.}
    If $z\in\Delta$ and $k|\tau|\le1$, then $|z-\rho|=|\tau|^2/4\le r_k$,
    so the segment $[\rho,z]$ lies in $\bar D(\rho,r_k)$, and
    \[
        \left|Q_k^{[m]}(z)-Q_k^{[m]}(\rho)\right|\le4B_m|z-\rho|=B_m|\tau|^2 .
    \]
    For the finitely many $k<k_1$, it suffices to bound $|(Q_k^{[m]})'(\xi)|$ on $\bar D(\rho,1)$.
\end{proof}

\subsection{Analytic continuation}

\begin{lemma}\label{lem:G-analytic}
For every $m\ge1$, the series $G^{[m]}(z)=\sum_{k\ge1}Q_k^{[m]}(z)$
converges uniformly on~$\Delta$, and locally uniformly on a
$\Delta$-domain at~$\rho$. Hence $G^{[m]}$ is analytic in a
$\Delta$-domain at~$\rho$, and $G^{[m]}(z)\to G^{[m]}(\rho)<\infty$ as
$z\to\rho$ in $\Delta$.
\end{lemma}

\begin{proof}
    The key ideas of the proof are as follows.
    Near $\rho$, the bounds of Lemma~\ref{lem:marked-scaling} already give uniform convergence on $\Delta$ (Step~1).
    At any other point $z_0$ of the circle $|z|=\rho$,
    the modulus of the linear coefficient $1-\tau$ of the error recurrence is less than $1$ (Step~2),
    so the error $e_k(z)$ becomes small and the factors $2zI_k(z)$ contracting (Step~3).
    The functions $Q_k^{[m]}(z)$ then decay geometrically near $z_0$ (Step~4),
    and a compactness argument concludes (Step~5).

\emph{Step 1: near $\rho$.}
    On $\bar\Delta$, uniform convergence follows from Lemma~\ref{lem:marked-scaling}(ii),
    $|Q_k^{[m]}(z)|\le C_m/k^2$.

\emph{Step 2: the linear coefficient away from $\rho$.}
    Consider $z_0$ with $|z_0|=\rho$ and $z_0\neq\rho$.
    Then $w_0:=1-4z_0$ satisfies $|w_0-1|=1$ and $w_0\ne0$,
    so $t_0:=\sqrt{w_0}$ has $\Re t_0>0$.
    Hence $|t_0-1|<|t_0+1|$, and since $|t_0-1|\cdot|t_0+1|=|w_0-1|=1$, we have
    \[
        |1-t_0|<1 .
    \]
    The function $\tau(z)$ is continuous near $z_0$, and so is $C(z)=(1-\tau)/(2z)$.

\emph{Step 3: a neighborhood where the factors contract.}
    Choose $\gamma\in(0,\frac14)$ with $|1-t_0|\le1-3\gamma$,
    and $k_0$ with $e_{k_0}(\rho)\le\gamma/2$ (Proposition~\ref{thm:e-asymp}).
    Since $e_{k_0}(z)=C(z)-I_{k_0}(z)$ has nonnegative coefficients
    (it counts the trees of height larger than $k_0$),
    $|e_{k_0}(z_0)|\le e_{k_0}(\rho)\le\gamma/2$.
    By continuity, there is a bounded neighborhood $V$ of $z_0$ on which
    \[
        |1-\tau(z)|\le1-2\gamma,
        \qquad
        |e_{k_0}(z)|\le\gamma,
        \qquad
        |z|\le\tfrac12 .
    \]
    On $V$, by induction on $k\ge k_0$,
    $|e_{k+1}(z)|=|e_k(z)|\cdot|1-\tau-ze_k(z)|\le|e_k(z)|(1-2\gamma+\gamma/2)$,
    so $|e_k(z)|\le\gamma$, and then
    \[
        \bigl|2zI_k(z)\bigr|=\bigl|1-\tau-2ze_k(z)\bigr|\le1-\gamma
        \qquad(k\ge k_0).
    \]

\emph{Step 4: geometric decay.}
    We show by induction on $m$ that there are constants $B_m$ such that
    $|Q_k^{[m]}(z)|\le B_m(1-\gamma)^k$ on $V$ for all $k$.
    Assume this holds for $i<m$.
    By~\eqref{eq:Q^[m]-def}, $x_k:=\sup_V|Q_k^{[m]}(z)|$ satisfies
    \[
        x_{k+1}\le(1-\gamma)x_k+\beta_m(1-\gamma)^{2k}
        \quad(k>k_0),
        \qquad
        \beta_m:=\frac12\sum_{i<m}B_iB_{m-i}.
    \]
    Iterating from $k_0+1$, we get, for $k>k_0$,
    \[
        x_k\le(1-\gamma)^{k-k_0-1}x_{k_0+1}
        +\beta_m\sum_{i=k_0+1}^{k-1}(1-\gamma)^{k-1-i}(1-\gamma)^{2i}
        \le(1-\gamma)^k\Bigl(\frac{x_{k_0+1}}{(1-\gamma)^{k_0+1}}
        +\frac{\beta_m}{\gamma(1-\gamma)}\Bigr),
    \]
    and the finitely many polynomials $Q_k^{[m]}(z)$ with $k\le k_0$ are bounded on $V$.
    This proves the claim, and hence $G^{[m]}(z)$ is analytic on $V$.

\emph{Step 5: compactness.}
    The points $z=\rho e^{i\beta}$ with $0<|\beta|$ small,
    together with a small neighborhood of each of them,
    belong to $\Delta$ (for such points $\arg(z-\rho)$ is close to $\pm\pi/2$).
    By compactness of the rest of the circle $|z|=\rho$,
    finitely many neighborhoods $V$ cover it,
    and their union with $\{|z|<\rho\}$ and $\Delta$ contains a $\Delta$-domain at~$\rho$.
\end{proof}

\subsection{Proof of Lemma~\ref{lem:fatou}}
\label{app:fatou}

    Let us start with the intuition behind the proof.
    The map $\mathcal{F}$ is a perturbation of the translation $\omega\mapsto\omega+1$,
    with error $1/(\omega-1)$.
    On the half-plane $\{\omega\in\mathbb{C}:\Re\omega>2\}$,
    the iterates therefore go to infinity at linear speed,
    and the correction series defining $\widetilde\Phi$ converges (part (ii)).
    On a sector around the positive real axis,
    we show that the iterates still grow linearly and avoid the branch cut of the logarithm,
    which gives the extension and the asymptotic expansion (part (iii)).
    For convenience, we denote the iterates by $\omega_n:=\mathcal{F}^n(\omega)$.

\emph{(i)}
    If $\Re\omega>1$, then $\Re(\omega-1)>0$,
    hence $\Re\frac1{\omega-1}>0$ and $\Re \mathcal{F}(\omega)>\Re\omega+1$.

\emph{(ii) Convergence.}
    For $|\omega|\ge2$,
    \[
        h(\omega)=\sum_{k\ge2}\Bigl(1-\frac1k\Bigr)\omega^{-k},
        \qquad\text{and so}\qquad
        |h(\omega)|\le\frac{2}{|\omega|^2}.
    \]
    For $\Re\omega>2$, the iterates $\omega_n$ satisfy $\Re\omega_n\ge\Re\omega+n>2+n$ by (i),
    so the series $\sum_nh(\omega_n)$ converges uniformly
    and $\widetilde\Phi$ is analytic on $\{\omega\in\mathbb{C}:\Re\omega>2\}$.

\emph{(ii) Abel equation.}
    Since $\omega$ and $\mathcal{F}(\omega)$ lie in the right half-plane, we have the relation
    $\log \mathcal{F}(\omega)-\log\omega=\log\frac{\omega}{\omega-1}=-\log(1-\frac1\omega)$, and
    \[
        \widetilde\Phi\bigl(\mathcal{F}(\omega)\bigr)-\widetilde\Phi(\omega)
        = \mathcal{F}(\omega)-\omega+\log\Bigl(1-\frac1\omega\Bigr)-h(\omega)
        = 1+\frac1{\omega-1}+\log\Bigl(1-\frac1\omega\Bigr)-h(\omega) = 1 .
    \]
    For $\Re\omega>1$, we set $\widetilde\Phi(\omega):=\widetilde\Phi(\mathcal{F}(\omega))-1$,
    which is consistent by the Abel equation and analytic by (i).

\emph{(iii) Linear growth in a sector.}
    Let $s:=\sin\eta$ and $R_\eta:=32/s^2$.
    For a complex $\omega$ satisfying $|\arg\omega|<\pi-\eta$ and a real $t\ge0$,
    elementary geometry gives
    \begin{equation}\label{eq:omega+t}
        |\omega+t|\ge s\max\bigl(|\omega|,t\bigr)\ge\frac s2\bigl(|\omega|+t\bigr).
    \end{equation}
    Let $|\omega|>R_\eta$ and $|\arg\omega|<\pi-\eta$, and write
    \[
        \omega_n=\omega+n+\varepsilon_n,
        \qquad
        \varepsilon_n=\sum_{j=0}^{n-1}\frac1{\omega_j-1}.
    \]
    Let us show by induction that
    \begin{equation}\label{eq:vareps-ineq}
        |\varepsilon_n|\le\frac s4\bigl(|\omega|+n\bigr).
    \end{equation}
    For $n=1$, it follows from the inequalities
    \[
        \bigl(|\omega|+1\bigr)|\omega-1| \ge
        \bigl(|\omega|^2-1\bigr) \ge
        \frac{32}{s^2}-1 > \frac{4}{s} .
    \]
    If \eqref{eq:omega+t} holds up to $n$,
    then $|\omega_j|\ge\frac s4(|\omega|+j)$ and $|\omega_j-1|\ge\frac s8(|\omega|+j)$ for $j\le n$,
    so that
    \[
        |\varepsilon_{n+1}|\le\frac8s\sum_{j=0}^n\frac1{|\omega|+j}
        \le\frac8s\cdot\frac{n+1}{|\omega|}
        \le\frac{8}{s|\omega|}\bigl(|\omega|+n+1\bigr)
        \le\frac s4\bigl(|\omega|+n+1\bigr).
    \]

\emph{(iii) Extension and expansion.}
    It follows from \eqref{eq:omega+t} and \eqref{eq:vareps-ineq} that
    $|\omega_n|\ge\frac s4(|\omega|+n)\ge2$ for all $n$.
    Therefore, the series defining $\widetilde\Phi$ converges locally uniformly on this sector, and
    \[
        \Bigl|\widetilde\Phi(\omega)-\omega+\log\omega\Bigr|
        \le\sum_{n\ge0}\frac{32}{s^2(|\omega|+n)^2}
        \le\frac{64}{s^2|\omega|}.
    \]

\emph{(iii) Abel equation.}
    It remains to check that $\log\omega_{n+1}-\log\omega_n=-\log(1-1/\omega_n)$
    with principal logarithms along the whole orbit.
    The iterates never meet the half-line $(-\infty,0]$: indeed
    \[
        \Im\mathcal{F}(\omega)=\Im\omega\,\left(1-\frac{1}{|\omega-1|^2}\right),
    \]
    and $|\omega_j-1|\ge\frac s8|\omega|>1$,
    so that $\Im\omega_n$ has the sign of $\Im\omega$ for every $n$,
    while $\omega_n\in\mathbb{R}_{>0}$ for all~$n$ if $\omega\in\mathbb{R}_{>0}$.
    Thus, $\omega_n$ and $\omega_{n+1}$ lie in the same open half-plane,
    or both on $(0,\infty)$, and, since $\omega_{n+1}/\omega_n=(1-1/\omega_n)^{-1}$,
    \[
        \log\omega_{n+1}-\log\omega_n=\log\frac{\omega_{n+1}}{\omega_n}
        =-\log\Bigl(1-\frac1{\omega_n}\Bigr).
    \]
    The computation of (ii) then applies at every point of the orbit, so that
    $\widetilde\Phi(\omega_{n+1})=\widetilde\Phi(\omega_n)+1$ for all $n\ge0$.
    On the intersection with $\{\omega\in\mathbb{C}:\Re\omega>2\}$
    the two definitions are given by the same formula; hence the extensions glue together.

\emph{Consequences.}
    Since $4/(x-x^2/4)=\mathcal{F}(4/x)$ and $\Re(4/x)>1$ if and only if $|x-2|<2$,
    the function $\Phi(x)=\widetilde\Phi(4/x)$ is analytic on this disc and satisfies the Abel equation. Finally, for $|\arg x|<\pi-\eta$, we have $\log(4/x)=\log4-\log x$, so (iii) gives
    \[
        \Phi(x)=\frac4x+\log x-\log4+O(|x|)
        \qquad(x\to0). \hfill\qed
    \]

\subsection{Refined tail of the distribution}
\label{app:refined-tail}

The error term of Theorem~\ref{thm:distribution-leaves}(iii) can be
improved to any negative power of~$m$.

\begin{proposition}\label{prop:refined-tail}
For every $p\ge1$, as $m\to\infty$,
\[
    K^{[m]} = 4^{-m}\Bigl(16-\frac4m\Bigr)+O\bigl(4^{-m}m^{-p}\bigr),
    \qquad\text{hence}\qquad
    \kappa^{[m]} = 4^{-m+1}\Bigl(1-\frac{1}{4m}\Bigr)+O\bigl(4^{-m}m^{-p}\bigr).
\]
\end{proposition}

\begin{proof}
We first outline the proof strategy. Beyond $\omega-\log\omega$, the Fatou coordinate has an
asymptotic expansion in powers of $1/\omega$ only. We construct its
truncation $\varphi_p$ explicitly by solving the Abel equation up to
order $p$ (Step~1), and we show that it approximates $\widetilde\Phi$ up
to $O(|\omega|^{-p-1})$ (Step~2). Since $1/\omega=(4-y)/16$ in the proof
of Theorem~\ref{thm:distribution-leaves}, the additional terms are
polynomials in $y$, which do not contribute to the coefficients of large
index (Step~3). We use the notation of Lemma~\ref{lem:fatou} and of its
proof, with a fixed $\eta\in(0,\pi/2)$, and write $t:=1/\omega$.

\emph{Step 1: a formal Fatou coordinate.} For $|\omega|>2$, we have
$\mathcal{F}(\omega)^{-1}=(\omega-1)/\omega^2=t(1-t)$, hence
\[
    \mathcal{F}(\omega)^{-j}-\omega^{-j}=t^j\bigl((1-t)^j-1\bigr)=-j\,t^{j+1}+O(t^{j+2})
\]
is a polynomial in $t$. Recall from the proof of Lemma~\ref{lem:fatou}
that, with principal logarithms,
\[
    \bigl(\mathcal{F}(\omega)-\log\mathcal{F}(\omega)\bigr)-\bigl(\omega-\log\omega\bigr)-1
    = h(\omega) = \sum_{k\ge2}\Bigl(1-\frac1k\Bigr)t^k .
\]
Therefore, for any coefficients $a_1,\dots,a_p$, the function
\[
    \varphi_p(\omega):=\omega-\log\omega+\sum_{j=1}^p\frac{a_j}{\omega^j}
\]
satisfies $\varphi_p(\mathcal{F}(\omega))-\varphi_p(\omega)-1=E_p(\omega)$, where
\[
    E_p(\omega):=h(\omega)+\sum_{j=1}^p a_j\bigl(\mathcal{F}(\omega)^{-j}-\omega^{-j}\bigr)
    = h(\omega)+\sum_{j=1}^p a_j\,t^j\bigl((1-t)^j-1\bigr)
\]
is a convergent power series in $t$ for $|t|<\frac12$. This identity
explains why $\varphi_p$ is an approximate Fatou coordinate, but below we
only use $E_p$ through its definition, which involves no logarithm. The coefficient of
$t^{j+1}$ in $E_p$ is $-j\,a_j$ plus a quantity depending only on
$a_1,\dots,a_{j-1}$. Therefore, we can choose $a_1,\dots,a_p$ successively so
that the coefficients of $t^2,\dots,t^{p+1}$ vanish (for instance,
$a_1=\frac12$). Then $E_p(\omega)/t^{p+2}$ is analytic for $|t|<\frac12$,
hence bounded for $|t|\le\frac18$, that is,
\[
    |E_p(\omega)|\le C_p|\omega|^{-p-2}
    \qquad(|\omega|\ge8).
\]

\noindent\emph{Step 2: comparison with $\widetilde\Phi$.} Set
$S_\eta:=\{\omega\in\mathbb{C}:\ |\omega|>R_\eta,\ |\arg\omega|<\pi-\eta\}$, let
$\omega\in S_\eta$, and let $\omega_n=\mathcal{F}^n(\omega)$. By the proof
of Lemma~\ref{lem:fatou}(iii), with $s=\sin\eta$,
\[
    |\omega_n|\ge\frac s4(|\omega|+n)\ge\frac8s\ge8
    \qquad\text{and}\qquad
    \widetilde\Phi(\omega)=\omega-\log\omega+\sum_{n\ge0}h(\omega_n).
\]
Since $\omega_{n+1}=\mathcal{F}(\omega_n)$, the definition of $E_p$ reads
$E_p(\omega_n)=h(\omega_n)+\sum_{j=1}^pa_j\bigl(\omega_{n+1}^{-j}-\omega_n^{-j}\bigr)$.
Summing over $0\le n<N$, the second part telescopes:
\[
    \sum_{n=0}^{N-1}E_p(\omega_n)
    =\sum_{n=0}^{N-1}h(\omega_n)+\sum_{j=1}^pa_j\bigl(\omega_N^{-j}-\omega^{-j}\bigr).
\]
Letting $N\to\infty$, since $\omega_N\to\infty$, we obtain
\[
    \sum_{n\ge0}E_p(\omega_n)
    =\sum_{n\ge0}h(\omega_n)-\sum_{j=1}^pa_j\omega^{-j}
    =\widetilde\Phi(\omega)-\varphi_p(\omega).
\]
By Step~1, $|E_p(\omega_n)|\le C_p(4/s)^{p+2}(|\omega|+n)^{-p-2}$, so that
\[
    \bigl|\widetilde\Phi(\omega)-\varphi_p(\omega)\bigr|
    \le C_p\Bigl(\frac4s\Bigr)^{p+2}\sum_{n\ge0}\frac{1}{(|\omega|+n)^{p+2}}
    \le\frac{C_p'}{|\omega|^{p+1}} .
\]
Thus, in $S_\eta$,
\[
    \widetilde\Phi(\omega)=\omega-\log\omega+\sum_{j=1}^pa_j\omega^{-j}+O\bigl(|\omega|^{-p-1}\bigr).
\]

\noindent\emph{Step 3: transfer.} As in the proof of
Theorem~\ref{thm:distribution-leaves}, with $\omega=16/(4-y)$, in
a $\Delta$-domain at $y=4$, we get
\[
    K(y) = \frac{16}{1-y/4}+4\log\Bigl(1-\frac y4\Bigr)+c_K
    +4\sum_{j=1}^pa_j\Bigl(\frac{4-y}{16}\Bigr)^{j}
    +R_p(y),
    \qquad R_p(y)=O\Bigl(\Bigl|1-\frac y4\Bigr|^{p+1}\Bigr).
\]
The function $R_p$ is analytic in this $\Delta$-domain, since all the
other terms are. The polynomial terms do not contribute to $[y^m]$ for
$m>p$, and the transfer theorem~\cite[Theorem~VI.3]{FS09} gives
$[y^m]R_p(y)=O(4^{-m}m^{-p-2})$. Therefore, for $m>p$,
\[
    K^{[m]}=16\cdot4^{-m}-4\cdot\frac{4^{-m}}{m}+O\bigl(4^{-m}m^{-p-2}\bigr),
\]
and the claim for $\kappa^{[m]}=K^{[m]}/4$ follows.
\end{proof}

\section{Moments of the limiting distribution}
\label{app:moments}

Let us introduce
\[
    H(y) := \sum_{m \ge 1} \kappa^{[m]} y^m
\]
the probability generating function of $P$, the number of leaf pairs (i.e., half the number of leaves) at the deepest level.  From the functional equation $K\!\big((1+y/4)^2\big) = K(y) + 4$ and the relation $K(y) = 4H(y)$, we obtain, for $|y|<4$, the normalized form
\begin{equation}
\label{eq:func_eq_G}
    H\left(\left(1+\frac{y}{4}\right)^2\right) = H(y) + 1.
\end{equation}
Since $H$ is analytic in the disc $|y|<4$ (Theorem~\ref{thm:distribution-leaves}), which contains $y=1$, the moments of $P$ are finite and given by the derivatives of $H$ at $y=1$.

\noindent\textbf{Mean.} Differentiating \eqref{eq:func_eq_G} with respect to $y$ gives
\begin{equation}
\label{eq:diff1}
    \frac{1}{2}\left(1+\frac{y}{4}\right) H'\!\left(\left(1+\frac{y}{4}\right)^2\right) = H'(y).
\end{equation}
Setting $y=0$ yields
\[
    \frac{1}{2} H'(1) = H'(0) = \kappa^{[1]}.
\]
Therefore,
\begin{equation}
\label{eq:mean}
    \mathbb{E}[P] = H'(1) = 2\kappa^{[1]}.
\end{equation}

\noindent\textbf{Second moment.} Differentiating \eqref{eq:diff1} once more gives
\[
    \frac{1}{8} H'\!\left(\left(1+\frac{y}{4}\right)^2\right)
        + \frac{1}{4}\left(1+\frac{y}{4}\right)^2 H''\!\left(\left(1+\frac{y}{4}\right)^2\right) = H''(y).
\]
Evaluating at $y=0$ gives
\begin{equation*}
     \frac{1}{8} H'(1) + \frac{1}{4} H''(1) = H''(0).
\end{equation*}
Now $H''(0) = 2\kappa^{[2]}$ and $H'(1)=2\kappa^{[1]}$ from \eqref{eq:mean}.  Hence,
\[
     H''(1) = 8\kappa^{[2]} - \kappa^{[1]}.
\]

\noindent\textbf{Variance.} The second factorial moment is $\mathbb{E}[P(P-1)] = H''(1)$, so
\[
    \operatorname{Var}(P)
        = 8\kappa^{[2]} + \kappa^{[1]} - 4(\kappa^{[1]})^2.
\]
Since the actual number of leaves at the deepest level is $L = 2P$, we obtain
\[
    \operatorname{Var}(L)
        = 32\kappa^{[2]} + 4\kappa^{[1]} - 16(\kappa^{[1]})^2.
\]
Using the numerical values from Figure~\ref{tab:distribution}, we compute
\begin{equation*}
    \mathbb{E}[L] = 4\kappa^{[1]} \approx 2.8037,
    \qquad\qquad
    \operatorname{Var}(L) \approx 2.0649.
\end{equation*}

\end{document}